\documentclass[12pt]{amsart}

\usepackage{amsfonts,amsmath,latexsym,amssymb,verbatim,amsbsy,amsthm}
 \usepackage{boondox-cal}
\usepackage{graphicx}
\usepackage{soul}
\usepackage{dsfont}
\usepackage{bm}
\usepackage{cancel}
\usepackage[top=1in, bottom=1in, left=1in, right=1in]{geometry}

\usepackage[dvipsnames]{xcolor}

\usepackage[colorlinks=true, pdfstartview=FitV, linkcolor=RoyalBlue,citecolor=ForestGreen, urlcolor=blue]{hyperref}

\numberwithin{equation}{section}

\renewcommand{\geq}{\geqslant}

\renewcommand{\leq}{\leqslant}

\theoremstyle{plain}
\newtheorem{THEOREM}{Theorem}[section]

\newtheorem{theorem}[THEOREM]{Theorem}

\newtheorem{lemma}[THEOREM]{Lemma}
\newtheorem{proposition}[THEOREM]{Proposition}

\theoremstyle{definition}

\newtheorem{definition}[THEOREM]{Definition}

\theoremstyle{remark}

\newtheorem{remark}[THEOREM]{Remark}
\newtheorem{claim}[THEOREM]{Claim}

\newcommand{\thm}[1]{Theorem~\ref{#1}}
\newcommand{\lem}[1]{Lemma~\ref{#1}}

\newcommand{\prop}[1]{Proposition~\ref{#1}}

\def \a {\alpha} %{\alpha} is occupied for the order of Lambda
\def \b {\beta}
\def \d {\delta}

\def \g {\gamma}

\def \e {\varepsilon}

\def \l {\lambda}
\def \L {\Lambda}
\def \n {\nabla}

\def \th {\theta}

\def \w {\omega}
\def \O {\Omega}
\def\bu{{\mathbf u}}
\def \ev{\kappa} % the eigenvalue
\def \meas{\mathcal{m}\!}

\def \bs {\boldsymbol{\sigma}}

\def \bj {{\bf j}}
\def \bk {{\bf k}}
\def \bx {{\mathbf x}}
\def \bxp {{\mathbf y}}  % rather than \bx^\prime
\def \by {{\mathbf y}}
\def \bu {{\mathbf u}}
\def \bv {{\mathbf v}}
\def \bw {{\mathbf w}}
\def \bz {{\mathbf z}}

\def \cL {\mathcal{L}}
\def \cC {\mathcal{C}}
\def \cM {{M}} %better notation than {\mathcal{M}}?

\def \cP {\mathbf{P}}   % better notation than {\mathcal{P}}?

\def \cF {\mathcal{F}}

\def \aL {{\mathcal L}}

\def\one{{\mathds{1}}}

\newcommand{\N}{\ensuremath{\mathbb{N}}}   %%% naturals
\newcommand{\Z}{\ensuremath{\mathbb{Z}}}   %%% integers
\newcommand{\R}{\ensuremath{\mathbb{R}}}   %%% reals
\newcommand{\T}{\ensuremath{\mathbb{T}}}   %%% torus
\newcommand{\E}{\ensuremath{\mathbb{E}}}   %%% torus

\def \lan {\langle}
\def \ran {\rangle}
\def \p {\partial}
\def \ra {\rightarrow}
\def \ss {\subset}
\def \bs {\backslash}

\newcommand{\steps}[2]{{\sc Step #1: {\bf #2}}.}

\DeclareMathOperator{\sgn}{sgn}

\def \loc {\mathrm{loc}}

\def \rmax  {\rho_+}
\def \rmin {\rho_-}
\def \rhomax  {\bar{\rho}}
\def \rhomin {\underline{\rho}}

\def \xmin {x_-}

\def \dx  {\, \mbox{\upshape{d}}x}
\def \dbx  {\, \mbox{\upshape{d}}\bx}
\def \dbxp  {\, \mbox{\upshape{d}}\by}  %rather than d\bx'
\def \dby  {\, \mbox{\upshape{d}}\by}
\def \dt  {\, \mbox{\upshape{d}}t}

\def \dy  {\, \mbox{\upshape{d}}y}
\def \dz  {\, \mbox{d}z}
\def \dr  {\, \mbox{d}r}
\def \ds  {\, \mbox{d}s}
\def \dw  {\, \mbox{d}w}
\def \dth  {\, \mbox{d}\th}
\def \dmu  {\, \mbox{d}\mu}

\def \dxi  {\, \mbox{d}\xi}
\def \dmeas {\, \mbox{d}\meas}

\def \dbw  {\, \mbox{d}\bw}
\def \dbz  {\, \mbox{d}\bz}
\def \ddt  {\frac{\mbox{d\,\,}}{\mbox{d}t}}

\def \dD  {\mbox{D}}
\def \dd {\mbox{\upshape{d}}} 

\def \aL {\mathcal{L}}
\def\bx{{\mathbf x}}
\def\pshi{\varphi} 

\begin{document}
%%%%%%%%%

\title[Emergent dynamics with short-range interactions]{ Topologically-based fractional diffusion and\\emergent dynamics with short-range interactions}

\author{Roman Shvydkoy}
\address{Department of Mathematics, Statistics, and Computer Science, 
University of Illinois, Chicago}
\email{shvydkoy@uic.edu}

\author{Eitan Tadmor}
\address{Department of Mathematics, Center for Scientific Computation and Mathematical Modeling (CSCAMM), and Institute for Physical Sciences \& Technology (IPST), University of Maryland, College Park}
\email{tadmor@umd.edu}

\date{\today}



\thanks{\textbf{Acknowledgment.} Research was supported in part by NSF grants 
	DMS16-13911, RNMS11-07444 (KI-Net) and ONR grant 
	N00014-1812465 (ET), and by NSF
	 grant DMS 1515705, Simons Foundation, and the College of LAS, UIC (RS). ET thanks the hospitality of Laboratoire Jacques-Louis Lions in Sorbonne University and its  support through ERC grant 740623 under the EU Horizon 2020. %research and innovation programme. 
	 RS thanks Cyril Imbert for useful consultations, and \'Ecole Normale Sup\'erieure for hospitality.
}

\vspace*{-1.1cm}

\begin{abstract} We introduce a new class of models for emergent dynamics. It is based on a new communication protocol which incorporates two main features:    short-range   kernels which restrict  the communication to local metric balls, and anisotropic communication kernels,  adapted to the local density in these balls, which form \emph{topological neighborhoods}. We prove flocking behavior --- the emergence of global alignment for  regular, non-vacuous solutions of the $n$-dimensional models based on short-range topological communication.   Moreover, global  regularity (and hence unconditional flocking) of the one-dimensional model is proved via an application of a  De Giorgi-type method. To handle  the \emph{non-symmetric} singular kernels that arise  with our topological communication, we develop a new analysis for \emph{local} fractional elliptic operators, interesting for its own sake,   encountered  in the construction of our class of models.
\end{abstract}

\maketitle
\setcounter{tocdepth}{1}%{1}
\vspace*{-1.1cm}
\tableofcontents

%%%%%%%%%%%%%%%%%%%%%%%
\section{Introduction and statement of main results}
%%%%%%%%%%%%%%%%%%%%%%%%
\subsection{Emergent dynamics -- long-range and short-range kernels}
A fascinating aspect of collective dynamics is self-organization, in which     higher order patterns emerge  from an underlying dynamics driven by short-range interactions. This type of collective dynamics is found in a wide variety of biological, social, and technological contexts. We investigate this phenomena in the context of canonical models for  flocking and swarming. A key feature in these models is \emph{alignment}, where a crowd described as a continuum with density $\rho(t,\bx): \R_+\times \R^n \mapsto \R_+$ aligns its macroscopic velocity, $\bu(t,\bx):\R_+\times \R^n \mapsto \R^n$,   over the local  neighborhoods ${\mathcal N}(\bx)$, 
\begin{equation}\label{e:main}
\left\{
\begin{split}
\rho_t + \n_\bx \cdot (\rho \bu) & = 0, \\
\bu_t + \bu \cdot \n_\bx \bu &= \int_{{\mathcal N}(\bx)}\phi(\bx,\by)(\bu(t,\by) - \bu(t,\bx))  \rho(t,\by)\dby.
\end{split} \right. 
\end{equation}
The dynamics   is subject to prescribed initial conditions, $(\rho_0,\bu_0)$, with two main configurations: either compactly supported density $diam\, \{supp \, \rho_0\} \leq D_0$ in $\R^n$, or  over the torus $\T^n$.
System \eqref{e:main} corresponds to  the large-crowd description of discrete crowd, consisting of $N \gg 1$ agents (of birds, insects, fish, robots, etc.) which align their microscopic velocities, $\{\bv_i(t)\}_{i=1}^N \in \R^n$,
\begin{equation}\label{eq:discrete}
\dot{\bv}_i= 
\sum_{j\in {\mathcal N}(\bx_i)} \phi\big(\bx_i(t),\bx_j(t)\big)(\bv_j(t)-\bv_i(t)),  \qquad \dot{\bx}_i=\bv_i
\end{equation}
Different models distinguish themselves with different choices of communication kernels, $\phi (\cdot, \cdot) \geq 0$, which dictate the neighborhoods ${\mathcal N}(\bx):=\{\by\, | \, \phi(\bx,\by)>0\}$. The most  notable examples found in the literature, \cite{Kur1975,Aok1982,Rey1987,VCBCS1995,Ben2005,CS2007a,CS2007b,MT2011}, employ  radial  kernels depending on the \emph{metric distance}
\begin{equation}\label{eq:geo}
\phi(\bx,\by)= \pshi(|\bx-\by|), 
 \end{equation}
that is, communication is taking place in balls, ${\mathcal N}(\bx)=B_{R_0}(\bx)$, where $R_0$ is the diameter of $supp \, \pshi$,
\begin{equation}\label{eq:geometric}
\left\{
\begin{split}
\rho_t + \n_\bx \cdot (\rho \bu) & = 0, \\
\bu_t + \bu \cdot \n_\bx \bu &= \int_{B_{R_0}(\bx)}\hspace*{-0.5cm}\pshi(|\bx-\by|)(\bu(t,\by) - \bu(t,\bx))  \rho(t,\by)\dby.
\end{split} \right. 
\end{equation} 

The communication kernels are in general  unknown: their approximate shape is either derived empirically \cite{StarFlag1,Bal2008,StarFlag2,StarFlag3,Co2011,CCGPS2012}, or learned from the data \cite{BPK2016,LZTM2019}, or  postulated based on phenomenological arguments, \cite{VZ2012,Bia2012,Bia2014}. 
Since the precise form of the communication kernel is in general not known, it is therefore imperative to understand how general $\pshi$'s  affect the large-time, large-crowd dynamics.
It is here that we make a distinction between  \emph{long-range} and \emph{short-range} interactions.

\medskip\noindent
{\bf Long-range interactions}. Here,  the support of $\pshi$ is large enough, $R_0 \gg1$,  so that every part of the crowd  is in direct communication with every other part. In particular, if $\pshi$ satisfies 
\begin{equation}\label{eq:long_range}
\mbox{a `fat tail' condition}: \qquad \int^\infty \!\!\pshi(r)\dr = \infty,
\end{equation}
 then $supp \, \rho(t,\cdot)$ remains within a finite diameter $D_\infty <\infty$, and consequently, the alignment dynamics \eqref{eq:geometric} enforces the the crowd  to `aggregate' around a limiting velocity, $\bu_\infty\in \R^n$. The flocking behavior in this case of long-range interactions is captured by the statement  ``smooth solutions must flock'', \cite{TT2014, HeT2017}, namely --- if $(\rho(t,\cdot), \bu(t,\cdot))  \in L^\infty \times W^{1,\infty}$ is  a  global strong solution of  \eqref{eq:geometric},\eqref{eq:long_range} subject to compactly supported initial data $(\rho_0,\bu_0)$, then, there exists $\eta>0$ (depending  on  $D_\infty$) such that $\bu(t,\cdot)$ flocks towards   a limiting velocity $\bu_\infty$,
\begin{equation}\label{eq:exp_decay}
 \max_\bx |\bu(t,\bx)-\bu_\infty| \lesssim e^{-\eta t} \rightarrow 0,  \qquad \bu_\infty=\frac{\cP_0}{\cM_0}, \quad  (\cM_0,\cP_0):= \int (1,\bu_0) \rho_0(\bx) \dbx.
\end{equation}
The unconditional flocking asserted  in \eqref{eq:exp_decay} is  rooted in the corresponding statement for the discrete dynamics   \eqref{eq:discrete}, with long-range interactions \eqref{eq:geo},\eqref{eq:long_range}, 
\cite{CS2007a,CS2007b,HT2008,HL2009, HHK2010,MT2014}. 

\smallskip
The conditional statement for long range interactions shifts the burden of proving their flocking behavior to the regularity theory. Here we make a further distinction between bounded and singular $\pshi$'s.

For \emph{bounded kernels}, global regularity in dimension $n=1,2$ holds if   the initial configuration satisfies a certain  threshold conditions, \cite{TT2014,CCTT2016,HeT2017}. Global regularity (and hence flocking behavior) of \eqref{eq:geometric} for any dimension  but for small data in higher order Sobolev spaces, $\|\bu\|_{H^{s+1}} < \e_0(\|\rho_0\|_{H^s})$ was proved in  \cite{HKK2014}.
The regularity and flocking behavior of \eqref{eq:geometric} with  \emph{singular kernels} $\pshi(r) = r^{-\b}$  was 
 studied in \cite{PS2017} for weakly singular kernels,  $0<\b<n$, and in \cite{ST3,ST1,ST2,DKRT2018} for strongly singular kernels, $\b =n+\a$, $0<\a<2$.  In the latter case, the system \eqref{eq:geometric} is endowed with a fractional parabolic diffusion structure which  enabled to prove, at least in the one-dimensional case, \emph{unconditional flocking behavior}, independent of any initial threshold.  We quote here our main result of \cite{ST3,ST2} which  will  be echoed in the statements of this present paper: for the system  \eqref{eq:geometric} with strongly singular kernel, $\pshi(r)=r^{-(n+\a)}, 0<\a<2$, on $\T$, any non-vacuous initial data gives rise to  a unique global solution,  $(\rho,u)\in L^\infty([0,\infty);  H^{s +\a} \times H^{s+1}), \ s\geq 3$, which converges  to a flocking traveling wave,
\[
	\|u(t,\cdot)-u_\infty\|_{H^s} +\|\rho(t,\cdot) - \rho_\infty(\cdot-tu_\infty) \|_{H^{s-1}} \lesssim  e^{-\eta t}, \quad t>0, \qquad 
	u_\infty:= \frac{P_0}{\cM_0}.
\]
The question of regularity  (and hence flocking) for strongly singular kernels $\pshi(r)=r^{-(n+\alpha)}$ in  dimensions $n>1$ is  open, with the  exceptions of  recent small initial data results in \cite{Sh2019} for
H\"older spaces, $|\bu_0-\bu_\infty|_\infty \lesssim (1+\|\rho_0\|_{W^{3,\infty}}+\|\bu_0\|_{W^{3,\infty}})^{-n}$ with $2/3<\alpha<3/2$, and in \cite{DMPW2019} for small Besov data $\|\bu_0\|_{B_{n,1}^{2-\alpha}} + \|\rho_0-1\|_{B_{n,1}^1} \leq \e$ with $\alpha \in (1,2)$.

\medskip\noindent
{\bf Short range interactions}. The class of singular kernels $\pshi(r)=r^{-\beta}$ offers a communication framework  which emphasizes  short-range  interactions  over long-range interactions, yet their global support still reflects global communication. In particular,  strongly singular kernels, $n< \beta< n+2$, demonstrates  hydrodynamic flocking  for thinner tails than those sought in \eqref{eq:long_range}, yet their infinite support still maintain   global direct communication over all $supp\, \rho(t,\cdot)$.\newline
This brings us back to the original question alluded to at the beginning, namely --- understanding self-organization  driven by   a \emph{purely local communication protocol}.  This is the question we address in our present work, in the  context of general alignment \eqref{e:main} with short-range singular communication kernels\footnote{Here and throughout $\one_S$ denote the characteristic function of a set $S$, and  $A\lesssim B$ means $A/B <C$ where $C$ is a fixed constant.}
\begin{equation}\label{eq:cutoff}
 \frac{\one_{|\bx-\by|<R_0}}{|\bx-\by|^{n+\a}}\lesssim \phi(\bx,\by) \lesssim \frac{\one_{|\bx-\by|<2R_0}}{|\bx-\by|^{n+\a}}, \quad 0<\a <2. 
\end{equation}
It provides a first fundamental step in our understanding of emergent phenomena in collective dynamics driven by short-range communication kernels.\newline 
It has been an open question whether the emergence of hydrodynamic flocking  survives the cut-off localization in \eqref{eq:cutoff}. The situation is analogous to the scenario  of discrete crowd  with  short range communication, \eqref{eq:discrete}, which may fail to flock due to finite-time loss of graph connectivity associated with the time-dependent adjacency matrix $\{\phi(\bx_i(t),\bx_j(t))\}$, \cite[sec. 2.2]{MT2014}. At the level of  hydrodynamic description \eqref{e:main}, lack of connectivity manifests itself as `thinning' of crowd density inside $supp \, \rho(t,\cdot)$, and eventually creating  vacuous sub-regions in which the flow does not exert any alignment on its neighborhood. In this case, the dynamics \eqref{e:main} is reduced  to inviscid Burgers-type  blowup \cite{Ta2017}, thereby demonstrating necessity of the no-vacuum assumption. 
This brings us to our first main result, asserting that  smooth non-vacuous solutions of alignment dynamics associated with a general class of \emph{short-range} singular  kernels, \eqref{eq:cutoff}, must flock.
\begin{theorem}[{\bf Smooth solutions must flock --- singular symmetric kernels}]\label{thm:local_must_flock}
\mbox{ }\newline
Let $(\rho(t,\cdot), \bu(t,\cdot))$  be a  global strong solution
of the  alignment dynamics \eqref{e:main} with short-range symmetric kernel \eqref{eq:cutoff}, over the torus $\T^n$.
Assume that
\begin{equation}\label{eq:away_from_vacuum}
\eta(t):=\int^t \rho^2_-(s) \ds \stackrel{t \rightarrow \infty}{\longrightarrow} \infty, \qquad \rho_-(t):=\min_\bx\rho(t,\bx).
\end{equation}
Then  there is convergence towards flocking $($with the average velocity $\displaystyle \bu_\infty=\frac{\cP_0}{\cM_0})$
 \begin{equation}
\int_{\T^n}|\bu(t,\bx)-\bu_\infty|^2\rho(t,\bx)\dbx \leq \frac{1}{2\cM_0}e^{-\eta(t)}.
 \end{equation}
\end{theorem}
Note that any positive lower bound on the density is impossible in the open space if finite mass is assumed. So,  periodic conditions are more natural for the settings. Compactness is also important for the proof which is presented in section \ref{sec:must_flock} below. \thm{thm:local_must_flock}  provides a general framework for the flocking of alignment dynamics driven by  short-range  singular communication kernels, under the assumption that the global solution is non-vacuous. Here, the precise decay rate of the density $\min \rho(t,\cdot)$ is at the heart of matter:  according to theorem \ref{thm:local_must_flock} unconditional flocking is  achieved under the  lower bound  
\begin{equation}\label{eq:sqrt}
\rho(t,\cdot) \gtrsim \frac{1}{\sqrt{1+t}}.
\end{equation} 
The difficulty is that  verification  of such \emph{apriori} lower bound seems out of reach. To address this difficulty, we now introduce a new {\em topological} short-range  communication protocol which tames the required decay rate of the density by adapting itself to  sub-regions with thinner densities. Moreover,  the new protocol fits to be more realistic in various behavioral experiments than the purely metric one as we will elaborate in the next section.

%%%%%%%%%%%%%%%%%%%%%%%%%%%%
\subsection{A new paradigm for collective dynamics -- topological kernels}
%%%%%%%%%%%%%%%%%%%%%%%%%%
We  introduce a new communication protocol based on the  principle that 
\emph{information between agents spreads faster  in regions of lower density}.
 To realize this principle we consider  communication kernel of the form
\begin{subequations}\label{eqs:dist}
\begin{equation}\label{eq:topo}
\phi(\bx,\by)=\pshi(|\bx-\by|)\times \frac{1}{\dd^{n}_\rho(\bx,\by)},
\end{equation}
which depends on two main features:

(i)   {\bf Metric distances}. $\pshi(r)$ reflects the dependence on metric distance in $\R^n$ (and respectively in $\T^n$), $r(\bx,\by)=|\bx-\by|$.  For the metric part of the communication, we use the  short-range singular kernel
\begin{equation}\label{eq:psi1}
\pshi(r)=\frac{h(r)}{r^{\a}}, \qquad  \one_{r<R_0} \lesssim h(r) \lesssim  \one_{r<2R_0}, \quad 0<\a<2.
\end{equation}
The smooth cut-off $h(r)$  guarantees  that communication is localized in  balls of radius $\leq 2R_0$.  

(ii) {\bf Topological distances}. For any two  parts of the crowd at two different locations $\bx,\by\in supp \, \rho(t,\cdot)$, we fix an  intermediate region of communication $\O(\bx,\by) \subset \R^n$ (or $\subset \T^n$). In the one-dimensional case, it is taken simply as the closed interval $\O(x,y)=[x,y]$; in the  multi-dimensional case, we choose a conical region outlined in section \ref{sec:region}. Then, $\dd_\rho(\bx,\by)$ reflects the dependence on the "mass" as a \emph{topological measure of a distance} between the crowd at $\bx$ and $\by$ -- specifically,
\begin{equation}\label{eq:drho} 
\dd_\rho(\bx,\by) := \left[\int_{\O(\bx,\by)}  \rho(t,\bz) \dbz \right]^{\frac{1}{n}}\ \  \text{with} \ \O(\bx,\by) \ \text{given in} \ \eqref{eq:region}.
\end{equation}
\end{subequations}

\begin{remark}[{\bf Why topological distances?}]
To motivate  the so-called topological distances \eqref{eq:drho} we refer to the underlying discrete setup \eqref{eq:discrete}. The discrete configuration of $N$ agents is captured by the empirical distribution $\mu_t(\bx,\bv)=\frac{1}{N}\sum_k\delta_{\bx_k(t)}(\bx)\otimes \delta_{\bv_k(t)}(\bv)$. Then $\mu_t(\Omega(\bx_i,\bx_j))$ amounts to  counting  the (discrete) crowd in the region of communication $\Omega(\bx_i,\bx_j)$, and we set the discrete distance to be
\[
\dd_N(\bx_i,\bx_j):=\big(\mu_t(\Omega(\bx_i,\bx_j))\big)^{\frac{1}{n}}=\left(\frac{\#\{\bx_k \, | \, \bx_k\in \O(\bx_i,\bx_j)\}}{N}\right)^{\frac{1}{n}}.
\]
The  dependence of the communication kernel \eqref{eq:topo} on $\dd^{-n}_N(\bx_i,\cdot)$ indicates that agent at $\bx_i$ places a strong preference of communication with its nearest agents, $\{\bx_j \, | \, \dd_N(\bx_i,\bx_j) \sim 
N^{-\frac{1}{n}}\}$, over the increased interference in communication with agents farther away, 
$\{\bx_j \, | \, \dd_N(\bx_i,\bx_j) \lesssim 1\}$. The net effect of probing low density neighborhoods using such singular  kernels is  communication  dictated  by the number of nearest agents rather than geometric proximity, \cite{Ha2013,BD2016,BD2017}. 
Letting $N\rightarrow \infty$  recovers the topological distance \eqref{eq:drho} in the continuum setup, $\displaystyle \dd_N(\bx,\by) \stackrel{N\rightarrow \infty}{\longrightarrow}\dd_\rho(\bx,\by)$.
 Thus, the corresponding alignment dynamics \eqref{e:main},\eqref{eqs:dist}  is a  continuum realization  of the same paradigm, namely ---  enhancing communication in regions of low density by  invoking the `density of closest neighbors' as the proper  continuum substitute  for the `number of  closest neighbors'.
 Accordingly, we refer to  $\dd_\rho(\bx_i,\bx_j)$  as \emph{topological (quasi-)distance}. 
This is consistent with the  established terminology  in experimental literature, which refers to such topological communication in flocking birds   \cite{StarFlag1,Bal2008,StarFlag2,StarFlag3} and in human  interaction in pedestrian dynamics \cite{Wa2018}.
\end{remark}

Noting that $\dd_\rho(\bx,\by)\gtrsim c(\rho)|\bx-\by|$,  it follows that $\phi(\bx,\by)$ is  singular  of order $n+\a$,  $\phi(\bx,\by) \lesssim \one_{|\bx-\by|\leq 2R_0}|\bx-\by|^{-(n+\a)}$. Thus, the  topological kernel \eqref{eqs:dist}  belongs to the general class short-range kernels \eqref{eq:cutoff}. It reflects  short-range communication (of diameter $\leq 2R_0$), maintaining finite amplitude $\{\by \ | \ \phi(\bx,\by) \gtrsim 1\}$ within active topological neighborhoods
\[
{\mathcal N}(\bx)=  \{ \by\in B_{\mbox{}_{2R_0}}(\bx)\ | \  \dd_\rho(\bx,\by) < c_0 \},
 \]
where $c_0$ is an empirical constant indicating perception ability of the agents. The kernel is  non-convolutive, and though $\phi$ is symmetric $\phi(\bx,\by) = \phi(\by,\bx)$, the full kernel that appears in the alignment term, $K(\bx,\by,t):=\phi(\bx,\by)\rho(\by)$, is not. The proper notion of the non-symmetric (strongly) singular alignment action on the right of \eqref{e:main}, $\cC_\phi(\rho, f)= \int \phi(\bx,\by)(f(\by)-f(\bx))\rho(\by)\dby$,  is discussed in section \ref{s:topo}. This brings us to our second main result.
\begin{theorem}[{\bf Flocking of  short-range topological kernels}]\label{t:alignrate-a}  Let $(\rho,\bu)$ be a  global smooth solution of the topological model  \eqref{e:main}, \eqref{eqs:dist} on $\T^n$. Assume that  the density $\rho(t,\cdot)$ satisfies, 
	\begin{equation}\label{e:ralg1}
	\rho(t,\bx) \geq  \frac{c}{1+t}.
	\end{equation}
	Then the solution aligns with $\bu_\infty$ with at least a root-logarithmic rate 
	\begin{equation}\label{e:alignrate-a}
	|\bu(t) - \bu_\infty |_\infty  \lesssim \frac{c}{\sqrt{\ln t}}.
	\end{equation}
\end{theorem}

The proof of theorem \ref{t:alignrate-a} --- given in section \ref{sec:short-topo} below, traces the propagation of information between the extreme values of (the components of) $\bu(t,\cdot)$,
which are most susceptible to breakup since they can no longer rely on distant communication. Instead, we  introduce a new method of sliding averages, in which we measure how far $\bu(t,\bx)$ deviates from its average over the \emph{local} balls $B(\bx,r), \ r\leq R_0$, using a density-weighted Campanato class. For some algebraic sequence of times $t_n \to \infty$, these deviations are proved to be small. At the same time, we show that overwhelmingly, $\bu(t,\bx)$ stays close to its extreme values near the critical points where these values are attained. To achieve this, we estimate the conditional probability of an unlikely event of $\bu$ being far from its extremes, in terms of the mass-measure $\dmeas_t = \rho \dbx$: it is here that  the topological-based alignment  in  \eqref{eq:topo} plays a key role. We end up with a  (finite) overlapping chain of non-vacuous balls to connect any two points and by chain estimates, the fluctuations of $\bu(t,\cdot)$  are shown to decay uniformly in time. This  explains the emergence of global alignment from short-range interactions which, to the best of our knowledge, is the first result  of its kind.

\medskip\noindent
In closing this section, a couple of remarks are in order.
\begin{remark}({\bf A comparison with Motsch-Tadmor scaling}). It is instructive to compare the  topological kernel \eqref{eqs:dist} which we rewrite as
\[
\phi(\bx,\by)=\pshi(|\bx-\by|)\times\frac{1}{\meas_t (\O(\bx,\by))},
\qquad \meas_t(\O):=\int_\O \rho(t,\bz)\dbz,
\]
with  the Motsch-Tadmor scaling \cite{MT2011} with local $\pshi(r)=\one_{r<R_0}$,
\[
\phi(\bx,\by)=\pshi(|\bx-\by|)\times\frac{1}{\meas_t (B_{\mbox{}_{R_0}}(\bx))}. 
\]
In the former, the pairwise interaction between two ``agents'' depends on the  density in an intermediate  region of communication; in the latter, the communication of each ``agent''  depends on how thin is the crowd in its own metric neighborhood. 
\end{remark}

%%%%%%%%%%%%%%%%%%%%%%%%%%%%%%%%%%
\subsection{Global regularity: drift-diffusion beyond symmetric kernels}
%%%%%%%%%%%%%%%%%%%%%%%%%%%%%%%%%
 As in the case of long-range communication, theorem \ref{t:alignrate-a} shifts the `burden' of proving flocking with short-range topological kernels to the question of existence: do \eqref{e:main},\eqref{eqs:dist} admit global smooth solutions with lower-bounded density $\rho(t,\cdot) \gtrsim (1+t)^{-1}$?  In section \ref{s:1D} which is at the heart of matter and occupies the bulk of this paper, we provide an affirmative answer for  the one-dimensional  model over $\T$, thus providing a first example of unconditional flocking. The question of non-vacuous global regularity in  dimension $n>1$ remains open.

To elaborate further on the required regularity of $(\rho,u)$, we note that both density and momentum equations in \eqref{e:main} fall under a general class of \emph{parabolic drift-diffusion} equations,
\[
u_t + {\mathbf b} \cdot \n_\bx u = \int K(\bx,\by,t) (u(\by) - u(\bx)) \dby  + f, 
\]
with (a priori) rough coefficients, ${\mathbf b}$, and with a proper singular local kernels
\[
\quad \frac{\one_{|\bx-\by|<R_0}}{|\bx-\by|^{1+\a}} \lesssim K(\bx,\by,t) \lesssim \frac{\one_{|\bx-\by|<2R_0}}{|\bx-\by|^{1+\a}}.
\]
Regularity theory for equations of this type  had a rapid development in recent years due to breakthroughs in understanding of the non-local structure of the fractional Laplacian, see Caffarelli et al \cite{CCV2011,CS2011}, Silverstre et al \cite{S2012,SS2016}, Mikulevicius and Pragarauskas \cite{MP2014}, and local jump processes in Chen et. al.  \cite{CKK2008} and the references therein. Any  of these regularity results requires, however, the symmetry of the kernel $K(\cdot,\cdot,t)$ which we lack in the present framework: thus, the velocity $\bu$  in our topological model \eqref{e:main} is governed by drift-diffusion associated with kernel $K(\bx,\by)=\phi(\bx,\by)\rho(\by)$: while $\phi(\cdot,\cdot)$ is symmetric, $K$ is not. Similarly, the same dynamics expressed in terms of the momentum, ${\mathbf m}:=\rho \bu$ or the density, consult  \eqref{e:m} and respectively \eqref{e:rparab}, encounters  the non-symmetric kernel $K(\bx,\by)=\phi(\bx,\by)\rho(\bx)$. 

Lack of symmetry in the $K$- kernels associated with  the topological communication \eqref{eqs:dist}  poses a fundamental difficulty  which prevents us from using the known results about the regularizing effect  in such transport-diffusion. Instead, we adapt the De Giorgi method to settle the H\"older regularity of $\rho(t,\cdot)$ in the critical case $\a=1$ (sec. \ref{s:degiorgi}), and employ fractional Schauder estimates to address the $\a>1$ case (sec. \ref{s:sch}). Together with the propagation of higher order regularity proved in sec. \ref{s:p}, we arrive at our third main regularity result stated below.

\begin{theorem}[{\bf Global regularity of  1D topological model}]\label{t:mainI}\mbox{ }\newline
 Consider the one-dimensional system \eqref{e:main} on $\T$ with short-range topological kernel \eqref{eqs:dist} and  singularity of order $1\leq \a<2$. Any non-vacuous initial data $(\rho_0,u_0)\in H^{s+\a}\times H^{s+1}$, $s\geq 3$, admits 
 a unique global in time solution,  $(\rho,u)$, in the class  
    \[
    \begin{split}
     \rho & \in C_w(\R^+; H^{s+\a})\cap L^2_\loc(\R^+; H^{s+1+\frac{\a}{2}}) \\
    u & \in C_w(\R^+; H^{s+1})\cap L^2_\loc(\R^+; H^{s+1+\frac{\a}{2}}), 
    \end{split}
    \]
    which flocks $|u(t,\cdot)-u_\infty|_\infty \rightarrow 0$.
\end{theorem}
Here, $C_w$ designates the space of weakly continuous function. Let us note that the density-enstrophy is expected to persist in a more natural, stronger  regularity space  $L^2_t H_x^{s+\a +\frac{\a}{2}}$ with $\a>1$, yet proving this would involve rather technical fractional energy estimates directly in $H^{s+\a}$, which we will postpone to future work. 
\begin{remark}
What distinguishes the 1D setup  is a conservation law, $e_t+(ue)_x=0$, of the first-order quantity $\displaystyle e=u_x+\int \phi(x,y)(\rho(y)-\rho(x))\dx$: while this is known for the metric kernels, $\phi=\pshi(|x-y|) $, \cite{CCTT2016,ST1,DKRT2018}, it is remarkable that  the same conservation  law still survives for the \emph{anisotropic} topological kernels $\pshi(|x-y|)\dd_\rho(x,y)$. In section \ref{s:e} we show that it enforces the parabolic character of the 1D mass equation $\rho_t+(u\rho)_x=0$ and in sec. \ref{s:density}, that it implies the  lower-bound $\rho(t,\cdot) \gtrsim (1+t)^{-1}$ sought in \eqref{e:ralg1}.
\end{remark}

\subsection{Notation} The following notation is used throughout the text:  $|f|_p$ stand for the classical $L^p$-norm, $1\leq p \leq \infty$, $\|f\|_X$ stands for all other norms such as $H^s$, etc, and $[f]_\g$, $0<\g<1$ stand for the H\"older semi-norm. The use of the following brackets is adopted:
\[
\lan f,g \ran = \int_{\T^n} f(\bx) g(\bx) \dbx,\quad 
\lan f,g \ran_\rho  =  \int_{\T^n} f(\bx) g(\bx) \rho(\bx) \dbx.
\]
We denote $\d_{\bz} f(\bx) = f(\bx+\bz ) - f(\bx)$. For  Sobolev spaces of fractional order, $H^{s}(\T^n)$, $0<s<1$, we always adopt the Gagliardo definition which states
\begin{equation}\label{e:GSdef}
	\|f\|_{H^s}^2 = \int_{\T^n} |\d_{\bz} f(\bx)|^2 \phi_s(\bz) \dbz,
\end{equation}
where 
\[
\phi_s(\bz) = \sum_{\bk \in \Z^n} \frac{1}{|\bz+2\pi \bk|^{n+2s}}.
\]
Considering $f$ periodically extended to $\R^n$ the above is the same as
\[
\|f\|_{H^s}^2 = \int_{\R^n} |\d_{\bz} f(\bx)|^2 \frac{\dbz}{|\bz|^{n+2s}}.
\]
We sometimes may use the latter for the benefit of a more explicity defined kernel.

%%%%%%%%%%%%%%%%%%%%
\section{The new protocol: short-range topological diffusion}
%%%%%%%%%%%%%%%%%%%%
In what follows we restrict ourselves to the periodic domain $\T^n$. This choice is motivated by the fact that the density in \eqref{e:main} quantifies parabolicity of the equation. With finite mass $\cM<\infty$ such parabolicity cannot be controlled uniformly on the open space.
In this section we elaborate on the basic ingredients  which are involved in the short-range, singular topological alignment model \eqref{e:main}, \eqref{eqs:dist},
\begin{subequations}\label{eqs:main_revisited}
\begin{equation}\label{eq:main_revisited}
\left\{
\begin{split}
\rho_t + \n \cdot (\rho \bu) & = 0, \\
\bu_t + \bu \cdot \n_\bx \bu &= \int_{\T^n}\phi(\bx,\by)(\bu(\by) - \bu(\bx))  \rho(\by)\dby, 
\end{split} \right. 
\end{equation}
where $\phi$ is the topological kernel given by
\begin{equation}\label{eq:topoker}
	\phi(\bx,\by) =\frac{h(|\bx-\by|)}{|\bx-\by|^{\a}}\times \frac{1}{\dd^n_\rho(\bx,\by)}, \quad \one_{r<R_0} \lesssim h(r) \lesssim  \one_{r<2R_0}, \quad 0<\a<2.
\end{equation}
\end{subequations}
Here, the first component of the kernel is quantified in terms of metric distance $|\bx - \by|$, the second involves the topological ``distance" $\dd_\rho(\bx,\by)$  between $\bx$ and $\by$,  defined by the mass located in the intermediate region of communication $\O(\bx,\by)$ 
\[
\dd_\rho(\bx,\by)=\left[\int_{\Omega(\bx,\by)}\rho(t,\bz)\dbz\right]^{\frac{1}{n}}.
\]
The region of communication enclosed between   $\bx$ and $\by$ is outlined in \ref{sec:region} below.  Observe that in absence of pressure each component $u$ of $\bu$ satisfies the maximum principle,
$\min u_0 \leq u(t,\cdot) \leq  \max u_0$, and that for all global regular solutions, $u\in L^1_{\loc} W^{1,\infty}$, the density remains non-vacuous, $\rho_0(x) > 0  \leadsto \rho(t,\bx) > 0 \text{ for all } t\geq 0$; hence we may assume that the density $\rho$ is a non-vacuous kinematic quantity satisfying
\begin{equation}\label{e:rbounds}
0 < c(t) \leq \rho(t,\bx) \leq C(t) <\infty, \quad \bx\in \T^n.
\end{equation}
Note that  although the distance function $\dd_\rho$  is not a proper metric (except for the one-dimensional case where it accumulates the mass along the interval $[x,y]$), it defines an equivalent topology on $\T^n$ such that $\dd_\rho(\bx,\by) \geq c |\bx-\by|$, and all the distances are bounded by the total mass $\cM$. 
 Moreover, since $\O(\bx,\by) = \O(\by,\bx)$, the topological distance is symmetric $\dd_\rho(\bx,\by) = \dd_\rho(\by,\bx)$.
%%%%%%%%%%%%%%%%%%
\subsection{Region of  communication}\label{sec:region}
%%%%%%%%%%%%%%%%%%
  The topological distance $\dd_\rho(\bx,\by)$ requires us to specify a domain of communication,  $\O(\bx,\by)$,
which is probed by agents located at $\bx$ and $\by$. In the one-dimensional case, it is simply the closed interval, $\O(x,y)=[x,y]$. In the multi-dimensional case, it is reasonably argued that the `intermediate environment' between agents could be an  $n$-dimensional region inside the ball enclosed by $\bx$ and $\by$, namely $B(\frac{\bx+\by}{2},r)$ with radius $r:=\frac{|\bx-\by|}{2}$.  For example, one  can simply set  $\O(\bx,\by)$ to be that ball. As we shall see below,  however, the   fine structure  of the local regions of communication, $\O(\bx_i,\bx_j)$, is important in order to retain unconditional flocking. To this end, we set a more restrictive \emph{conical} region $\O(\bx,\by)$, see Figure~\ref{f:omega}. First, we consider two basic locations $\bx =  (-1,0,...,0)$ and $\by= (1,0,...,0)$ and set  the region of revolution generated by a parabolic arch connecting $\bx$ and $\by$:
\[
\O_0 := \{\bz=(a,\bz_-) \ \big| \  |\bz_-| <1 - a^2, -1\leq a\leq 1\}.
\]
For an arbitrary pair of points $\bx,\by\in \R^n$, let $\O(\bx,\by)$ denote the region scaled and translated from $\O_0$:
\begin{equation}\label{eq:region}
\O(\bx,\by) := \{\bz  \ \big| \  |\bz-\bz_-| <1 -r^2a^2  \}, \qquad r=\frac{|\bx-\by|}{2},
\end{equation}
where $\bz_-: =\bz(a)$ is the projection of $\bz$ on the diameter
$\{\bz_-(a)= \frac{\bx+\by}{2}+\frac{a}{2}(\by-\bx), \ -1\leq a \leq 1\}$
connecting $\bx$ and $\by$.
 
\begin{figure}
	\includegraphics[width=5in]{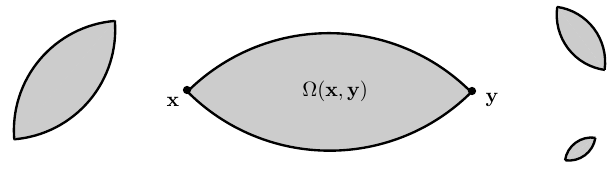}
	\caption{Communication domains between agents}\label{f:omega}
\end{figure}
Observe that at  the tips, $\O(\bx,\by)$ has the opening of $\frac{\pi}{2}$. For subsequent analysis,it   can be replaced by any angle $<{\pi}$, calibrated according to a particular application\footnote{Thus, for example,  \eqref{eq:region} can be enlarged to $\displaystyle \O(\bx,\by) := \{\bz  \ \big| \  |\bz-\bz_-|^{\g} <1 -r^2a^2  \}$ for any $0<\g<2$.}.  It is crucial, however, that the region of communication  is not locally smooth near the tips $\bx,\by$, see Claim~\ref{c:geom} below, which excludes the ball $B(\frac{\bx+\by}{2},r)$ with  conical opening of $90^\circ$.

%%%%%%%%%%%%%%%%%%%%%%%%%%%%%%%
\subsection{Topological kernels and the operators they define}\label{s:topo}
%%%%%%%%%%%%%%%%%%%%%%%%%%%%%%%%%
A distinctive feature of the alignment term on the right of \eqref{eq:main_revisited}  is that it admits a (formal) commutator structure \cite{ST1} 
\[
\int_{\T^n}\phi(\bx,\by)(\bu(\by) - \bu(\bx)) \rho(\by) \dby =\aL_\phi(\rho\bu)-\aL_\phi(\rho)\bu: = \cC_\phi(\bu,\rho),
\]
where  $\cL_\phi$ is the  integral operator given formally by
\begin{equation}\label{eq:aL}
\aL_\phi(f):=p.v.\int_{\T^n} \phi(\bx,\by)(f(\by)-f(\bx))\dby.
\end{equation}
Strong solutions to the system \eqref{e:main} satisfy energy equality
\begin{subequations}\label{eqs:eneq}
\begin{equation}\label{e:eneq}
\ddt \int_{\T^n} \rho |\bu|^2 \dbx = - \int_{\T^n} \phi(\bx,\bxp) |\bu(\bx) - \bu(\bxp)|^2 \rho(\bx)\rho(\bxp)\dbx\dbxp,
\end{equation}
which will be a key component in establishing alignment. We note on passing that  in view of the symmetry of the kernel $\phi$, we have conservation of mass and momentum:
\[
\cM(t) = \int_{\T^n} \rho(t,\bx) \dbx\equiv\cM_0, \quad \cP(t) = \int_{\T^n} \rho\bu(t,\bx)  \dbx\equiv\cP_0.
\]
Hence, the rate of decay of the energy of the left of \eqref{e:eneq} is the same rate of decay of the  fluctuations
\begin{equation}\label{eq:dfluc}
\ddt \int_{\T^{2n}} |\bu(t,\bx)-\bu(t,\by)|^2\rho(t,\bx)\rho(t,\by)\dbx\dby = 2 M_0 \ddt \int_{\T^n} \rho |\bu|^2 \dbx.
\end{equation}
\end{subequations}
Since  we have the Galilean invariance
$\bu \ra \bu(\bx+t{\mathbf U},t) - {\mathbf U}$ and $\rho \ra \rho(\bx+t{\mathbf U} , t)$ we may assume that $\cP(t)=\cP_0=0$.
 
We note that a proper care has to be given in order to properly define the singular integral operators $\cL_\phi f(\bx)$ and the corresponding commutator 
\begin{equation}\label{eq:comm}
\cC_\phi(f,g)  = \int_{\T^n} \phi(\bx,\by)(f(\by)-f(\bx))g(\by)\dby,
\end{equation} 
for strongly singular kernels $\a\geq 1$. Our immediate goal below is therefore to develop formal definitions and initial  facts about the operator $\cL_\phi$ in multi-D settings (more details specific for 1D situation will follow in Section~\ref{s:Leib}). Due to the non-convolutive and anisotropic nature of the kernel, most of the standard facts do not apply and will need to be readdressed.  Our plan is to define $\cL_\phi f$ as a distribution first. Then we state a formal justification of pointwise evaluations of $\cL_\phi f(\bx)$ and the commutator 
$\cC_\phi(f,g)$, so as to justify the fundamental  bookkeeping of energy/enstrophy fluctuations in \eqref{eqs:eneq}. Technicalities of the proofs will be collected in the Appendix.

\begin{definition}[{\bf The topologically-based fractional diffusion}]
	With the kernel given by \eqref{eq:topoker} we define an operator $\cL_\phi : H^{\a/2} \to H^{-\a/2}$ by the following action: for any $f\in H^{\a/2} $ and $g\in H^{\a/2}$
	\begin{equation}\label{e:Ldef}
	\lan    \cL_\phi f , g \ran = - \frac12 \int_{\T^{2n}} \phi(\bx,\by)(f(\bx) - f(\by))(g(\bx) - g(\by))  \dby \dbx.
	\end{equation}
\end{definition}
Note that formally such action could be obtained from \eqref{eq:aL}, if \eqref{eq:aL} made sense pointwise, by the usual symmetrization. Clearly, from the Gagliardo definition of $H^{\a/2}$, \eqref{e:GSdef}, we have
\[
|\lan    \cL_\phi f , g \ran| \lesssim \|f\|_{H^{\a/2}}  \| g \|_{H^{\a/2}}.
\]
Due to the symmetry of the kernel, the operator $\cL_\phi$ is clearly self-adjoint, and its range is in $H^{-\a/2}_0$ (here subscript $0$ means mean-free distributions). By the standard operator theory this implies the following statement.
\begin{lemma}\label{l:invert} The restricted operator $\cL_\phi : H_0^{\a/2} \to H_0^{-\a/2}$ is invertible.
\end{lemma}
\begin{proof}
Clearly, $c_0 \|f\|_{H_0^{\a/2}}^2\leq -\lan    \cL_\phi f , f \ran \leq C_0 \|f\|_{H_0^{\a/2}}^2$. Hence, $\|\cL_\phi f\|_{H^{-\a/2}} \geq c \|f\|_{H^{\a/2}}$ which shows that the operator has closed range and is injective. If the range is not all of $H^{-\a/2}_0$, then there is a $g \in H_0^{\a/2}$ for which $\lan    \cL_\phi f , g \ran  = 0$ for all $f\in H^{\a/2}$. Taking $f=g$ we arrive at a contradiction.  Thus, $\cL_\phi$ is invertible.
\end{proof}

In what follows we will need to be able to evaluate the action of the operator pointwise.  In the range $0<\a<1$ such evaluation presents no problem as long as $f\in C^1$. The rigorous argument goes by ``unwinding" the symmetric defining formula \eqref{e:Ldef}. To demonstrate it, let us denote by $L_\phi f(\bx)$ the integral on the right hand side of \eqref{eq:aL}. Clearly, $L_\phi f\in C(\T^n)$. Let us fix a point $\bx_0\in \T^n$. Let $g$ be the standard non-negative Friedrichs' mollifier supported on the ball of radius $1$. Denote $g_\e = \frac{1}{\e^n}g((\bx-\bx_0)/\e)$. It suffices to show that
\[
\lan    \cL_\phi f , g_\e \ran \ra L_\phi f(\bx_0).
\]
Since for $0<\a<1$, $L_\phi f(x)$ is a continuous function we can break up the integral without ambiguity:
\[
\begin{split}
\lan    \cL_\phi f , g_\e \ran &= - \frac12 \int_{\T^{2n}} (f(\bx) - f(\by))(g_\e(\bx) - g_\e(\by)) \phi(\bx,\by) \dby \dbx \\
&=  \int_{\T^{2n}} (f(\by) - f(\bx))g_\e(\bx) \phi(\bx,\by) \dby \dbx = \lan L_\phi f, g_\e \ran \ra L_\phi f(\bx_0).
\end{split}
\]

The higher case $1\leq \a<2$ is more subtle. 
Let us show that when $\rho$ and $f$ are smooth, the 
element $\cL_\phi f \in H^{-\a/2}$ gains regularity. Formally, this first step is necessary to even discuss pointwise values $\cL_\phi f(\bx)$. So, let us make the following observation:
\begin{equation}\label{e:derMD}
\n_\bx \dd_\rho(\bx+\bz,\bx) = \frac{1}{\dd^{n-1}_\rho(\bx+\bz,\bx)}\int_{\O(\bx+\bz,\bx)} \n \rho(\by) \dby = \int_{\p \O(\bx+\bz,\bx)} \vec{\nu} \rho(\by) \dby.
\end{equation}
Clearly, if $|\n \rho|_\infty <\infty$, then $|\n_\bx \dd_\rho(\bx+\bz,\bx) | \leq C |\n \rho|_\infty |\bz|$ with $C$ depending on a standing hypothesis on the density \eqref{e:rbounds}. Next, we rewrite the defining formula \eqref{e:Ldef} in terms of 
  the difference operator 
$\d_\bz f(\bx) := f(\bx+\bz) - f(\bx)$,
\[
\begin{split}
\lan    \cL_\phi f , g \ran &= -  \frac12 \int_{\T^{2n}} \d_\bz f(\bx)\d_\bz g(\bx) \phi(\bx,\bx+\bz)\dbx \dbz\\
& = -\frac12  \int_0^1 \int_{\T^{2n}} \d_\bz f(\bx) \n g(\bx+ \th \bz) \cdot \bz \ \phi(\bx,\bx+\bz) \dbx \dbz \dth.
\end{split}
\]
Integrating by parts and recalling \eqref{eqs:dist}, $\phi(\bx+\bz,\bx)=\dfrac{h(\bz)}{|\bz|^\a}\times \dd^{-1}_{\rho}(\bx+\bz,\bx)$, we obtain
\[
\begin{split}
\lan    \cL_\phi f , g \ran&=  \frac12  \int_0^1 \int_{\T^{2n}} \d_\bz \n f(\bx) \cdot \bz g(\bx+ \th \bz) \phi(\bx,\bx+\bz)  \dbx \dbz \dth \\
& +  \frac12  \int_0^1 \int_{\T^{2n}} \d_\bz f(\bx) g(\bx+ \th \bz) \d_\bz \rho(\bx) \frac{ \n \dd_\rho(\bx+\bz,\bx) \cdot \bz}{ |\bz|^\a \dd^{n+1}_\rho(\bx+\bz,\bx)} h(\bz) \dbx \dbz \dth.
\end{split}
\]
Note that the singularity of the kernels appearing inside both integrals is of order $n+\a-1$ now. With additional use of smoothness of other quantities we obtain
\[
| \lan    \cL_\phi f , g \ran| \lesssim ( \| f \|_{C^2} + \| f \|_{C^1} \| \rho \|_{C^1}) | g|_{\infty}.
\]
This is of course not an optimal bound, but it shows that the regularity of $\cL_\phi f$ improves. One can continue in similar fashion. Assuming $g = \p_x^k h$, for some $h\in L^\infty$, one obtains
\[
 | \lan    \cL_\phi f , \p_x^k h  \ran| \lesssim C(\|f\|_{C^{k+2}},\|\rho\|_{C^{k+1}}) |h|_{\infty}.
\]
Thus, $ \cL_\phi f \in (C^{-k})^* \ss C^{k-\e}$, for any $\e>0$. 

Lemmas~\ref{l:topoL} and \ref{l:topoC} stated in the Appendix make a formal justification for representation formulas  \eqref{eq:aL} and \eqref{eq:comm} which are to be understood in the principal value sense. They come with estimates that will be crucial in the proof of the global regularity in 1D, see Section~\ref{s:1D}. 

In what follows the density function $\rho$, of course depends on time, and so does the kernel. However, we will suppress the time variable for notational brevity.

\subsection{Leibnitz rules and coercivity}\label{s:Leib}
%%%%%%%%%%%%%%%%%%%%%%%%%%%
In this section we develop basic product rules and coercivity estimates for the operator $\cL_\phi$. We restrict ourselves to the one-dimensional case both for notational simplicity and for its use in the proof of regularity  asserted in \thm{t:mainI}. 

We start with basic product formulas for the derivative of $\aL_\phi f$ provided
$f$ and $\rho$ are smooth. 

First, let us observe that \eqref{e:derMD} in 1D case takes a simple form:
\begin{equation}\label{e:derd}
\p_x  \dd_\rho(x+z,x) = \d_z \rho(x) \sgn(z).
\end{equation}
Formally the Leibnitz rule reads
\begin{equation}\label{e:derL}
(\aL_\phi f)' = \aL_\phi (f') + \aL_{\phi'} f,
\end{equation}
where
\[
\aL_{\phi'} (f) = -\int_{\T}  \frac{h(z)}{|z|^\a \dd^2_\rho(x,x+z)}\d_z \rho(x) \sgn(z) \d_z f(x) \dz.
\]
The symmetric kernel $\phi'$ is of the same order $1+\a$. So, we can make sense of the integral in the same way as we did for $\aL_\phi$.   Rigorous justification of \eqref{e:derL} follows by proving \eqref{e:derL} in its weak formulation. So, for any $g\in C^\infty$, we have
\[
\begin{split}
\lan    (\cL_\phi f)' , g\ran & =  - \lan \cL_\phi f, g' \ran \\
& = \frac12 \int \d_z f(x) \d_z g'(x) \phi(\dd_\rho(x+z,x),z) \dx \dz \\
& = - \frac12 \int \d_z f'(x) \d_z g(x) \phi(\dd_\rho(x+z,x),z) \dx \dz \\
&- \frac12 \int \d_z f(x) \d_z g(x) \p_x \phi(\dd_\rho(x+z,x),z) \dx \dz\\
& = \lan    \cL_\phi( f') , g\ran +  \lan    \cL_{\phi'} f, g\ran .
\end{split}
\]
Continuing in the same fashion we obtain
\begin{equation}\label{e:LRn}
(\cL_\phi f)^{(n)} = \sum_{k=0}^n \frac{n!}{k!(n-k)!} \cL_{\phi^{(k)}} f^{(n-k)}.
\end{equation}

We can now discuss coercivity property of the operator $\cL_\phi$. In tune with the fact that $\cL_\phi$ puts $\a$-derivatives on $f$, it is natural to expect that $\cL_\phi f \in H^s$ if and only if $f \in H^{s+\a}$. For the topological kernels, however, this is a delicate result, details of which are presented in the following lemma. 
\begin{lemma}
	For any $s \geq 3$, and $1\leq \a <2$ one has the following bounds
	\begin{equation}\label{e:coers}
	\begin{split}
	\|\mathcal{L}_{\phi}  f \|_{{H}^s}^2 & \lesssim \|f\|^2_{{H}^{s+\alpha}} + \|f\|^N_{{H}^{s+\frac{\alpha}{2}}}+\|\rho\|^N_{{H}^{s+\frac{\alpha}{2}}} +1\\
	\|\mathcal{L}_{\phi}  f \|_{{H}^s}^2  & \gtrsim  \|f\|^2_{{H}^{s+\alpha}} -  \|f\|^N_{{H}^{s+\frac{\alpha}{2}}}-\|\rho\|^N_{{H}^{s+\frac{\alpha}{2}}} -1
	\end{split}
	\end{equation}
	where $N = N(n,\a,s)$, and $\lesssim$ denotes up to a constant depending on $(\min \rho)^{-1}$ and $\max \rho$. 
\end{lemma}
\begin{proof}
	According to \lem{l:kcomm} the commutator satisfies 
	\[
	| (\cL_\phi f)^{(s)} -  \cL_\phi (f^{(s)}) |_2^2  \lesssim \|f\|^N_{{H}^{s+\frac{\alpha}{2}}}+\|\rho\|^N_{{H}^{s+\frac{\alpha}{2}}} +1.
	\]
To deduce that we simply observe that all the dependancies on 	$|\rho'|_\infty,|f'|_\infty,\ldots , |f^{(k-1)}|_\infty, |\rho^{(k-1)}|_\infty$ translate into ${H}^{s+\frac{\alpha}{2}}$-norms by the Sobolev embedding. So, it remains to estimate the top term $|\cL_\phi (f^{(s)})|_2^2$. 

Let us denote for simplicity $f^{(s)} = g$. We ``freeze" the density in the topological distance as follows 
\[
\cL_\phi g(x) = \frac{1}{\rho(x)} \int_{\mathbb{T}} \frac{h(|z|)}{|z|^{1+\alpha}} \d_z g(x)\dz+ \int_{\mathbb{T}} \frac{h(|z|)}{|z|^{1+\alpha}}\left( \frac{1}{\frac{1}{|z|}\int_{[0,z]} \rho(x+\xi) \dxi} - \frac{1}{\rho(x)} \right) \d_z g(x)\dz = J_1 + J_2.
\]
The first integral represents the truncated fractional Laplacian. We clearly have
\[
|J_1|_2^2 \sim \|g\|^2_{{H}^{\alpha}}.
\]
As to $J_2$ we estimate
\[
\left| \frac{1}{\frac{1}{|z|}\int_{[0,z]} \rho(x+\xi) \dxi} - \frac{1}{\rho(x)} \right| \lesssim |z||\n \rho|_\infty,
\]
and with that
\begin{align*}
	|J_2(x)| \lesssim |\n \rho|_\infty \int_{\mathbb{T}} \frac{h(|z|)}{|z|^{\alpha}} |\d_z g(x)|\dz  &=  |\n \rho|_\infty \int_{\mathbb{T}} \frac{h(|z|)}{|z|^{\frac{\alpha+1}{2}}} |\d_z g(x)|\frac{\dz}{|z|^{\frac{\a-1}{2}}} \\
	& \lesssim  |\n \rho|_\infty \|g\|_{H^{\a/2}} \lesssim \|f\|^N_{{H}^{s+\frac{\alpha}{2}}}+\|\rho\|^N_{{H}^{s+\frac{\alpha}{2}}}.
\end{align*}
Putting together the obtained estimates proves the lemma. 
\end{proof}

%%%%%%%%%%%%%%%%%%%%%%%%%%%%
\section{Smooth solutions must flock}\label{sec:must_flock}
%%%%%%%%%%%%%%%%%%%%%%%%%%%%

The goal of this section will be to prove that any global, non-vacuous smooth solution to 
the topological model \eqref{e:main} aligns to its average velocity vector 
$\bu_\infty$ which can be determined from the conservation of momentum and mass: $\bu_\infty  = \cP_0 / \cM_0$.

%%%%%%%%%%%%%%%%%%%%%%
\subsection{Flocking for local symmetric  kernels}
%%%%%%%%%%%%%%%%%%%%%%
Let us first cast the question of flocking in the general settings \eqref{eq:cutoff} which includes both metric  \eqref{eq:geo} and topological kernels \eqref{eq:topo}, as well as other singular $\phi$'s localized along the diagonal. In other words, at this point we do not specify any fine structure of the kernel near the singularity. 
We recast the fundamental energy balance relation \eqref{eqs:eneq}, valid for \emph{any} singular symmetric kernel  via our definition \eqref{e:Ldef}:
\begin{equation}\label{eq:bulk}
\begin{split}
\ddt \int_{\T^{2n}} |\bu(t,\bx)-&\bu(t,\by)|^2\rho(t,\bx)\rho(t,\by)\dbx\dby
=-2M_0\int_{\T^{2n}} \langle \cC_\phi(\bu,\rho),\bu \rangle_\rho \dby \\
& =
-2M_0\int_{\T^{2n}} \phi(\bx,\by)|\bu(t,\bx)-\bu(t,\by)|^2\rho(t,\bx)\rho(t,\by)\dbx\dby. 
\end{split}
\end{equation}
The  main technical aspect of deriving a proper Gr\"onwall differential inequality from \eqref{eq:bulk} consist of obtaining lower-bounds of the \emph{enstrophy} on the right hand side of \eqref{eq:bulk} for  short-range $\phi$'s.

It is clear that a \emph{necessary condition} for flocking $|\bu(t,\cdot)-\bu_\infty| \rightarrow 0$ requires the density to be bounded away from vacuum, or else the flow may break apart into two or more separate `islands', traveling in their own velocity which is disconnected from the influence of others. Indeed, when $\rho(\cdot, t)$ vanishes on a compact set, the momentum equation \eqref{e:main} is reduced to the pressureless Burgers system $\bu_t+\bu\cdot\nabla_\bx\bu=0$ which in turn leads to a finite-time blow-up, see \cite{Ta2017}. Precisely how far from vacuum the density must be in order to fulfill an alignment dynamics for general local kernels $\phi$ is asserted in \eqref{eq:away_from_vacuum}. This brings us to the proof of our first main result.\newline
 \begin{proof}[{\bf Proof of theorem \ref{thm:local_must_flock}}]
We begin by setting up the general Hilbert structure for a variational formulation of the problem. 
Let us denote by $L^2_{\rho}$ the space of $L^2(\T^n)$-fields $\bu$ with scalar product given by 
\[
\lan \bu,\bv \ran_{\rho} = \int_{\T^n} \bu(\bx)\cdot \bv(\bx) \rho(t,\bx) \dbx.
\]
Note that the metric of the space $L^2_{\rho}$ changes in time.
Next, we consider the family of eigenvalue problems parametrized by time:
we seek eigenpairs, $\ev(t)$ and  $\bu(t,\cdot)$, 
\begin{equation}\label{e:second}
\int_{\T^n} \phi(\bx,\bxp)(\bu(\bxp)-\bu(\bx))\rho(t,\bxp)\dbxp= \ev(t)\bu(\bx), \qquad \bu \in {\mathcal U}^\alpha_\rho:=L^2_{\rho}\cap H^{\a/2}.
\end{equation}
Note that the left hand side is precisely the action of the commutator $\cC_\phi(\bu,\rho)$ which -- for any fixed smooth $\rho$ and any symmetric kernel satisfying \eqref{eq:cutoff}, maps $H^{\a/2}$ into $H^{-\a/2}$. Moreover, the symmetric definition of $\aL_\phi$ \eqref{e:Ldef} yields that $-\cC_\phi(\bu,\rho)$ is non-negative, $-(\cC_\phi(\bu,\rho),\bu)\geq0$. Hence $\ev_1=0$ is the minimal eigenevalue corresponding to the constant solution $\bu\equiv {\mathbf 1}$, and this allows us to seek the \emph{second} minimal eigenvalue as a solution to the variational problem\footnote{By symmetry $\overline{\bu}=\bu_\infty:=\cP_0/\cM_0$ but we keep the separate notation of $\overline{\bu}$ to signify orthogonality to the $0$-eigen-space spanned by ${\mathbf 1}$.}
\begin{equation}\label{e:lam1}
\ev_2(t) = \inf_{\bu\in {\mathcal U}^\a_\rho} \frac{-\lan \cC_\phi(\bu-\overline{\bu},\rho),\bu-\overline{\bu}\ran_\rho}{|\bu-\overline{\bu}|^2_{L^2_\rho}}, \quad \overline{\bu}:=\frac{\int \bu \rho}{\int \rho} \ \ \mbox {so that} \ \ \langle\bu-\overline{\bu},{\mathbf 1}\rangle_\rho=0 
\end{equation}
or --- stated explicitly in terms of $\displaystyle |\bu-\overline{\bu}|^2_{L^2_\rho}=\frac{1}{2M_0}\int_{\T^{2n}}|\bu(\by)-\bu(\bx)|^2\rho(\bx)\rho(\by)\dbx\dby$,
\begin{equation}\label{e:lam1a}
\ev_2(t) = 2M_0 \times \inf_{\bu\in {\mathcal U}^\a_\rho}\frac{\displaystyle\int_{\T^{2n}}  \phi(\bx,\bxp)|\bu(\bxp)-\bu(\bx)|^2\rho(t,\bxp)\rho(t,\bx)\dbx \dbxp}{\displaystyle\int_{\T^{2n}}|\bu(\bx)-\bu(\by)|^2\rho(t,\bx)\rho(t,\by)\dbx\dby}.
\end{equation}
Since the numerator with $\phi(\bx,\by)\simeq |\bx-\by|^{-(n+\a)}\one_{r<R_0}(|\bx-\by|)$ is equivalent for the $H^{\a/2}$-norm, the existence follows classically by compactness. This links the enstrophy on the right of \eqref{eq:bulk}  to $\ev_2(t)$, in complete analogy to the discrete case in which the coercivity of the discrete enstrophy is dictated by the  Fiedler number, consult \cite[sec 2.2]{MT2014}.

We can now state an alignment estimate in terms of the shrinking $L^2_\rho$-diameter of the velocity, given by 
\begin{equation}\label{e:Vdiam}
V_2[\bu,\rho](t):=\int_{\T^{2n}} |\bu(t,\bx)-\bu(t,\bxp)|^2\rho(t,\bx)\rho(t,\bxp)\dbx\dbxp.
\end{equation}
By \eqref{eq:bulk}, \eqref{e:lam1a} we have
\begin{equation}\label{e:L2V}
\ddt V_2[\bu,\rho](t) \leq - \ev_2(t) V_2[\bu,\rho](t).
\end{equation}

 The implication of \eqref{e:L2V} is of course the bound
\begin{equation}\label{e:Vexp}
2\cM_0\int_{\T^n}|\bu(t,\bx)-\bu_\infty|^2\rho(t,\bx)\dbx =  V_2[\bu,\rho](t)  \leq V_2[\bu_0,\rho_0] \exp\left\{ - \int_0^t \ev_2(s) \ds \right\}. 
 \end{equation}
Consequently, the solution aligns in the $L^2_\rho$-distance sense if $\displaystyle \int_0^\infty \ev_2(s) \ds = \infty$. It is here that we use the assumed lower-bound  on the density, $\rho(t,\cdot)\gtrsim \rho_-(t)$, the assumed singularity of our kernel
  $\phi(\bx,\bxp) \gtrsim |\bx-\bxp|^{-(n+\a)}\one_{|\bx-\bxp|<R_0}$ and by the uniform upper-bound of the density, $|\bu-\overline{\bu}|_{L^2_\rho} \lesssim |\bu|_{L^2}$, in order to bound the spectral gap 
\begin{equation}\label{e:lam3}
\ev_2(t) \geq c\rho^2_-(t) \inf_{\bu\in {\mathcal U}^\a_\rho}  \frac{\displaystyle\int_{|\bx-\bxp|<R_0} \frac{|\bu(\bx)-\bu(\bxp)|^2}{|\bx-\bxp|^{n+\a}}\dbx\dbxp}{|\bu|^2_2}, \qquad c:=\frac{2M_0}{C^2}.
\end{equation}
Technically, the infimum still depends on time since it is taken over the orthogonal complement of the line spanned by $\rho(t)$, denoted $[\rho(t)]$, in the classical  $L^2(\T^n)$. We now have to show that this infimum still stays bounded away from zero. Geometrically this is due to the fact that the space $[\rho(t)]^\perp$ does not come close to the span of constants $\R^n$ in the sense of Hausdorff  distance. It is more straightforward to argue by contradiction, however. 

 Suppose there is a sequence of times $t_k\in \R^+$, and $\bu_k \in  L^2_{\rho(t_k)}\cap H^{\a/2}$ such that $|\bu_k|_2 = 1$  yet the homogeneous local $H^{\a/2}$-norm tends to zero:
\begin{equation}\label{e:Havanish}
 \int_{|\bx-\bxp|<R_0} \frac{|\bu_k(\bx)-\bu_k(\bxp)|^2}{|\bx-\bxp|^{n+\a}}\dbx\dbxp \to 0.
\end{equation}
 Note that the latter, in particular, implies compactness of the sequence $\{\bu_k\}_k$ in $L^2$. Hence, up to a subsequence, $\bu_k \to \bu_*$ strongly in $L^2$ and weakly in $H^{\a/2}$. By the weak lower-semicontinuity and \eqref{e:Havanish}, we conclude that $\bu_* \in \R^n$ is a constant field, with $|\bu_*|=1$ due to $|\bu_k|_2 \to |\bu_*|_2$.
 
 At the same time, since $\rho(t)>0$ and $\int \rho(t_k,\bx) \dbx = \cM_0$, there exists a 
 weak$^*$ limit of a further subsequence $\rho(t_k) \to \mu$, where $\mu$ is a positive Radon measure on $\T^n$ with non-trivial total mass $\mu(\T^n) = \cM_0$ (since $M_0=\lan \rho(t_k) , 1 \ran \to \lan \mu, 1 \ran$).  We now reach a contradiction if we prove the limit
 \[
0= \int_{\T^n} \bu_k (\bx) \rho(t_k,\bx) \dbx \to \int_{\T^n} \bu_* \dmu = \cM_0\bu_* .
\]
To prove the claimed limit note that the assumed uniform upper-bound of the density implies
 \[
\begin{split}
  \int_{\T^n} \bu_k (\bx) &\rho(t_k,\bx) \dbx - \int_{\T^n} \bu_* \dmu(\bx) \\
  & = \int_{\T^n} \bu_k (\bx) \rho(t_k,\bx) \dbx - \cM_0\bu_* 
=   \int_{\T^n} (\bu_k (\bx) - \bu_*) \rho(t_k,\bx) \dbx,
\end{split}
\] 
and the latter is clearly bounded by $C |\bu_k - \bu_*|_2 \to 0$.
We conclude that 
\[
\int \ev_2(s)\ds \geq c \eta(t) =c\int^t \rho^2_-(s) \ds\rightarrow \infty,
\]
and the result follows from \eqref{e:Vexp}.
\end{proof}

%%%%%%%%%%%%%%%%%%%%%%%%%%
\subsection{Flocking with short-range topological kernels}\label{sec:short-topo}
%%%%%%%%%%%%%%%%%%%%%%%%%%%

We now turn our attention to the topological communication kernel \eqref{eqs:dist} and prove our main result, which improves the general \thm{thm:local_must_flock} to include a more natural condition on the density.

\begin{proof}[{\bf Proof of theorem \ref{t:alignrate-a}}]
Let us fix a coordinate $i$ and aim to prove \eqref{e:alignrate-a} for $u_i$. We denote $u = u_i$ for notational simplicity. Using the Galilean invariance we can lift $u$ if necessary and assume  that $u(t) >0$. 
Note that the extrema of $u(t)$, denoted $u_+(t)$ and $u_-(t)$, are monotonically decreasing and increasing, respectively.

We will make frequent use of the mass measure denoted
\[
\dmeas_t =  \rho(t,\bz) \dbz.
\]
\medskip
\noindent
\steps{1}{flattening near  extremes}  Let $\bx_+(t)$ be a point of maximum for $u(t,\cdot)$ 
and $\bx_-(t)$ a point of minimum. Let us fix a time-dependent $\d(t) >0$ to be specified later, and consider the sets
\[
G^+_\d(t) = \{ u<  u_+(t)(1-\d(t))\}, \quad G^-_\d(t) = \{ u> u_-(t)(1+\d(t))\}.
\]
The effect of flattening is expressed in terms of conditional expectations of the above sets in the balls $B(\bx_\pm(t), R_0)$ with respect to the mass measure. Let us denote
 \[
\E_t[A|B] = \frac{\meas_t(A\cap B)}{ \meas_t(B)}.
\]
We show that
\begin{equation}\label{e:G}
\int_0^\infty \d(t) \E_t[G^\pm_\d(t) | B(\bx_\pm(t), R_0)]  \dt <\infty.
\end{equation}
We focus on the '$+$' case, as the '$-$' case is entirely similar. To this end, let us compute the equation pointwise 
at the critical point $(t,\bx_+(t))$ utilizing the Rademacher Theorem: $(\p_t u) (t,\bx_+(t))= \p_t u_+(t)$ a.e.,
    \[
    \p_t u_+(t) =  \int \phi(\bx_+(t),\by) (u(\by) - u_+(t)) \rho(\by)\dby.
    \]
    At point $(\bx_+(t),t)$ we estimate on the alignment term with the use of the following observation:
    \begin{equation}\label{e:dn2f}
    c_0\frac{\one_{r<R_0}(|\bx - \by|)}{\dd^{n}_\rho(\bx,\by)} \leq  \phi(\bx,\by),
    \end{equation}
    for some $c_0>0$. Thus, we have
    \[
    \begin{split}
     - \p_t u_+(t) &= \int  \phi(\bx_+,\by)  (u_+(t) - u(\by)) \rho(\by)\dby\\
      & \geq c_0  \int_{B(\bx_+,R_0)} \frac{1}{\dd^{n}_\rho(\bx_+,\by)} (u_+(t) - u(\by)) \rho(\by)\dby, \\
        & \geq  \frac{c_0}{\meas_t( B(\bx_+(t), R_0)) } \int_{G^+_\d(t) \cap B(\bx_+(t), R_0)}  \hspace*{-1.5cm}(u_+(t) - u(\by)) \rho(\by)\dby \qquad (\mbox{since} \ \O(\bx_+,\by) \ss B(\bx_+,R_0))\\
    & \geq   \frac{c_0 \d(t)  u_+(t) }{\meas_t( B(\bx_+(t), R_0)) }  \int_{G^+_\d(t) \cap B(\bx_+(t), R_0)}  \hspace*{-0.5cm}\rho(\by)\dby \\
    &= c_0 \d(t)  u_+(t)  \E_t[G^+_\d(t) | B(\bx_+(t), R_0)].
    \end{split}
    \]
The result follows by integration:
\[
c_0 \int_0^\infty \d(t) \E_t[G^+_\d(t) | B(\bx_+(t), R_0)]   \dt  \leq  \ln \frac{u_+(0)}{\lim_{t\to \infty} u_+(t)} \leq \ln \frac{u_+(0)}{u_-(0)}.
\]

\medskip
\noindent
\steps{2}{Campanato estimates}  On this next step we obtain proper Campanato estimates that measure deviation of $u$ from its average values in terms of global enstrophy. 

We denote the averages with respect to mass-measure by
\[
u_{\bx,r} =  \frac{1}{\meas_t(B(\bx,r))} \int_{B(\bx,r)} u(t,\bz) \dmeas_t(\bz).
\]
Fix $\bx_*\in \T^n$. By H\"older inequality, we have the following estimate:
\[
\int_{|\bx-\bx_*|<\frac{r}{10}} | u(\bx) - u_{\bx_*,r}|^2 \rho(\bx) \dbx \leq \int_{\substack{|\bx-\bx_*|<\frac{r}{10}\\|\bxp-\bx_*|<r}  }\frac{1}{\meas_t(B(\bx_*,r))}  | u(\bx) - u(\bxp)|^2 \rho(\bx) \rho(\bxp)\dbxp\dbx
\]

At this point we recall that the communication domain $\O(\bx,\bxp)$ in \eqref{eq:region} has corner tips of opening $\frac{\pi}{2}$ degrees. Hence, we can make the following geometric observation.
\begin{claim} \label{c:geom} If $|\bx-\bx_*|<\frac{1}{10}r$ and $|\bxp-\bx_*|<r$, then $\O(\bx,\bxp) \ss B(\bx_*,r)$.
\end{claim}
In other words if $\bxp$ is in a ball and $\bx$ is close enough to the center of that ball then the domain $\O(\bx,\bxp)$ is entirely enclosed in the ball also, see Figure~\ref{f:geom}. This implies that $\meas_t(B(\bx_*,r)) \geq \meas_t(\O(\bx,\bxp)) = \dd^{n}_\rho(\bx,\bxp)$. We thus can further estimate, with the use of \eqref{e:dn2f},
\[
\begin{split}
\int_{|\bx-\bx_*|<\frac{r}{10}} | u(\bx) - u_{\bx_*,r}|^2 \rho(\bx) \dbx
& \leq \int_{|\bx-\bxp|<\frac{11}{10}r}   \frac{1}{\dd^{n}_\rho(\bx,\bxp)}  | u(\bx) - u(\bxp)|^2 \rho(\bx) \rho(\bxp)\dbxp\dbx\\
& \leq  \int_{\T^2}  \phi(\bx,\bxp) | u(\bx) - u(\bxp)|^2 \rho(\bx) \rho(\bxp)\dbxp\dbx.
\end{split}
\]

\begin{figure}
	\includegraphics[width=3.0in]{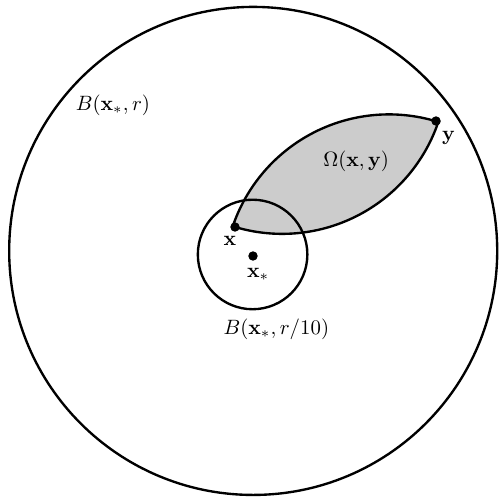} 
	\caption{$\Omega(\bx,\bxp)$ is trapped in the outer ball if $\bx$ is close to the center.}
	\label{f:geom}
\end{figure}

The energy balance \eqref{eq:bulk} (see also \eqref{eqs:eneq})  yields the space-time  bound
 on the  (components of) enstrophy on the right 
\[
\int_0^\infty \int_{\T^{2n}} \phi(\bx,\bxp) |u(\bx) - u(\bxp)|^2 \rho(\bx)\rho(\bxp)\dbx\dbxp \leq \frac12 \int_{\T^n} \rho_0 |\bu_0|^2 \dbx < \infty,
\]
hence we conclude with a time bound on the Campanato semi-norm,
\begin{equation}\label{eq:key}
\int_0^\infty [u]_{\rho}^2 \dt < \infty, \qquad [u]_{\rho}^2 := \sup_{\bx_*\in \T^n, r<\frac{R_0}{2}}\int_{|\bx-\bx_*|<\frac{r}{10}} | u(\bx) - u_{\bx_*,r}|^2 \rho(\bx) \dbx.
\end{equation}
Combined with \eqref{e:G} we have obtained
\[
I = \int_0^\infty  \Big(  \d(t) \E_t[G^\pm_\d(t) | B(\bx_\pm(t), R_0)]  + [u(t)]_{\rho}^2  \Big) \dt  < \infty.
\]
Clearly, for $A = e^{2I}$ we have
\[
\int_T^{T^A} \frac{\dt }{t\ln t}  = 2I \ \text{ for all } \ T>0.
\]
Hence, for any $T>1$ we can find a $t\in [T, T^A]$ such that 
\begin{equation}\label{e:Ct}
\begin{split}
 & [u(t)]_{\rho}^2  < \frac{1}{t \ln t} \\
 & \E_t[G^+_\d(t) | B(\bx_+(t), R_0)]  + \E_t[G^-_\d(t) | B(\bx_-(t), R_0)]   < \frac{1}{ \d(t) t \ln t}
\end{split}
\end{equation}
In view of the assumed lower bound on the density this implies in particular that
\begin{equation}\label{e:camp}
\sup_{\bx_*,\, r<\frac{R_0}{2}}\int_{|\bx-\bx_*|<\frac{r}{10}} | u(\bx) - u_{\bx_*,r}|^2 \dbx \leq \frac{1}{\ln t}.
\end{equation}

\medskip
\noindent
\steps{3}{sliding averages} Let us assume that $t\in [T, T^A]$ is a time fixed above.  We will now reconnect the two averages $u_{\bx_+,r}$ and $u_{\bx_-,r}$ sliding along the line connecting $\bx_+$ and $\bx_-$, and show that the variation of those averages is small.

\begin{figure}
	\includegraphics[width=4in]{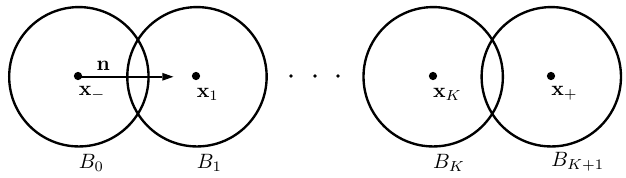} 
	\caption{}\label{f:balls2}
\end{figure}

 Denote the direction vector ${\mathbf n} = \frac{\bx_+ - \bx_-}{|\bx_+ - \bx_-|}$ and define a sequence of \emph{overlapping} balls, $B_k = B(\bx_k, \frac{r}{10}), \ k=0,\ldots,K$, with centers given by
$\bx_k = \bx_- + \frac{19r}{100} k {\mathbf n}$,  starting at $\bx_-$ and ending, with $K = [\frac{|\bx_+ - \bx_-|}{19r/100}]$, at  $\bx_{K+1} = \bx_+$,
see Figure~\ref{f:balls2}. 

Chebychev inequality, followed by \eqref{e:camp} applied to the ball centered at $\bx_*=\bx_0$, yields 
\[
|\{ \bx \in B_0 \cap B_1:  |u(\bx) - u_{\bx_0,r}| > \eta \} | \leq \frac{1}{\eta^2} \int_{B_0} |u(\bx) - u_{\bx_0,r}|^2 \dbx \leq  \frac{1}{\eta^2 \ln t}.
\]
We now fix scale $r := R_0/4$. Noticing that  $|B_k \cap B_{k+1}| =c R^n_0$ for all $k\leq K$,  and some a dimensional $c>0$, 
we  set $\displaystyle \eta = \frac{2}{\sqrt{c_0 R_0^{n} \ln t}}$ so that
\[
|\{ \bx \in B_0 \cap B_1:  |u(\bx) - u_{\bx_0,r}| > \eta \} | \leq \frac14 |B_0 \cap B_1|.
\]
Applying the same argument  to  the variation  around the  averaged value $u_{\bx_1,r}$, centered  at $\bx_*=\bx_1$, we obtain
\[
|\{ \bx \in B_0 \cap B_1:  |u(\bx) - u_{\bx_1,r}| > \eta \} | \leq   \frac14 |B_0 \cap B_1|.
\]
Consequently  the complements of the two sets must have a point in common in $B_0\cap B_1$:
\[\{ \bx \in B_0 \cap B_1:  |u(\bx) - u_{\bx_0,r}| \leq  \eta \} \cap
\{ \bx \in B_0 \cap B_1:  |u(\bx) - u_{\bx_1,r}| \leq  \eta \} \neq \emptyset,
\]
which implies that
\[
|u_{\bx_0,r} - u_{\bx_1,r}| \leq 2 \eta.
\]
Continuing in the same manner we obtain the same bound for all consecutive averages:
\[
|u_{\bx_k,r} - u_{\bx_{k+1},r}| \leq 2 \eta.
\]
Hence,
\begin{equation}\label{e:2aves}
|u_{\bx_-,r} - u_{\bx_+,r}| \leq 2 (K+1) \eta \lesssim \frac{1}{\sqrt{\ln t}}.
\end{equation}
Note that $K \leq 400\pi/R_0$, so it is bounded by an absolute constant. 
Furthermore, in view of  \eqref{e:Ct}, we can estimate
\[
\begin{split}
u_{\bx_+,r} & \geq \frac{1}{\meas_t(B(\bx_+,r))} \int_{B(\bx_+,r) \bs G_\d^+} u_+(t) (1-\d(t)) \dmeas_t \\
&\geq u_+(t)(1-\d(t)) ( 1 -  \E_t[G^+_\d(t) | B(\bx_+(t), R_0)]  ) \geq u_+(t)(1-\d(t)) \left( 1 - \frac{1}{ \d(t) t \ln t} \right).
\end{split}
\]
Hence,
\[
u_+(t) - u_{\bx_+,r}(t) \lesssim \d(t) +\frac{1}{ \d(t) t \ln t} \lesssim  \frac{1}{ \sqrt{ t \ln t}}
\]
if we set  $ \d(t) = \frac{1}{ \sqrt{ t \ln t}}$.  A similar estimate holds for the bottom average.  In conjunction with \eqref{e:2aves} these imply
\[
|u_+(t) - u_-(t)| \lesssim \frac{1}{\sqrt{\ln t}}.
\]
To conclude the proof we note that by the maximum principle 
\[
|u_+(T^A) - u_-(T^A)| \lesssim \frac{1}{\sqrt{\ln t}} \sim   \frac{1}{\sqrt{\ln (T^A)}}.
\]
Since $T$ is arbitrary this finishes the proof.

\end{proof}

%%%%%%%%%%%%%%%%%%%%%%%%
\section{Global well-posedness in 1D}  \label{s:1D}
%%%%%%%%%%%%%%%%%%%%%%%%
In this section we develop a more complete theory of one-dimensional topological models, and provide the proof of \thm{t:mainI}. In 1D the system takes form
\begin{equation}\label{e:main1D}
\left\{
\begin{split}
\rho_t + (\rho u)_x & = 0, \\
u_t + u u_x &= \cC_\phi( u,\rho), \ \ \phi(x,y)=\frac{h(|x-y|)}{|x-y|^{\a}} \times \frac{1}{\dd_\rho(x,y)}
\end{split} \right. \qquad (t,x)\in \R_+ \times \T,
\end{equation}
where 
\[
\dd_\rho(x,y)=\left|\int_x^y \rho(t,z)\dz\right|.
\]
The distinct feature of the one dimensional models with convolution metric kernels $\phi(x-y)$ is an extra conservation law:
\begin{equation}\label{e:e}
e_t + (ue)_x = 0, \qquad e := u_x + \aL_\phi  \rho.
\end{equation}
The derivation of the conservative ``$e$''-equation is straightforward with  either smooth or singular \emph{radial} kernels,  \cite{CCTT2016,ST1}.
It plays a key role in the regularity and hence unconditional flocking  of the 1D alignment with metric-based communication, \cite{CCTT2016,ST1,ST3}. Its role as a measure of disorder of the limiting flock was explored in \cite{LS-entropy}. A priori, there is no reason for \eqref{e:e} to hold in our case:  the derivation of such law stumbles upon the difficulty that  the operator $\cL_\phi$ does \emph{not} commute with derivatives.  Nevertheless, \emph{it is remarkable that  the law \eqref{e:e} still survives for anisotropic topological kernels}. To make our analysis rigorous we need to develop calculus of the operator $\cL_\phi$ and collect several analytical facts before we can proceed. This will be done in Section~\ref{s:Leib}.

Once we justify \eqref{e:e}, we can proceed in section \ref{s:e} to the regularity of the 1D solution along the lines of \cite{ST1,ST2}. Since the topological kernels lack  translation invariance, we need to revisit the question of propagation of regularity, section \ref{s:p} and H\"{o}lder regularization of the density in sections \ref{s:sch} and \ref{s:degiorgi}.

The proof will be split into several stages. First, before we even embark into technicalities of the argument, we develop necessary tools to work with the operator $\cL_\phi$ itself. It will be done in the next section.  Second, we establish a priori estimates on the density that are necessary to sustain uniform parabolicity and conclude the alignment, see Section~\ref{s:density}. Third, we prove a propagation of regularity result, \prop{p:global}, which states that if one can propagate some modulus of continuity of the density, then one can propagate any higher order regularity for both $u$ and $\rho$.  Fourth, we show how to gain a H\"older modulus of continuity from several sources. In the case $1<\a<2$ we reduce the problem to a known Schauder estimate for fractional singular operators. For the case $\a=1$, we employ the DeGiorgi method along the lines of Caffarelli, Chan, and Vasseur work \cite{CCV2011} with significant upgrades related to the presence of a drift, source, and asymmetry of the kernel involved. We also treat the system as truly nonlinear, see also \cite{GIMV2017}, and  highlight scaling properties of the system which become very important, see \eqref{e:urresc}-\eqref{e:rescale}.   

Finally, the alignment claim  follows directly from \thm{t:alignrate-a}. Indeed, the lower bound on the density \eqref{e:ralg1}  requires the rate  which will be established for any regular solutions in \lem{l:rhobounds} below. 

%%%%%%%%%%%%%%%%%%%%%%%%%%

First we note that the existence and uniqueness of local solutions of \eqref{e:main1D} can be deduced from the estimates we perform below when treated as a priori. We state the result here for our future reference.
\begin{theorem}\label{t:lwp}
 Let $1\leq \a <2$ and $s \geq 3$. 	For any initial data $u_0\in H^{s+1}(\T)$, $\rho_0 \in H^{s+\a}(\T)$,   with no vacuum $\rho_0(x) >0$ there exists a unique solution to the system \eqref{e:main1D} on  a time interval $[0,T)$ on which it will remain non-vacuous and belonging to the class
	\begin{equation}\label{e:class}
	\begin{split}
	u & \in C_w([0,T), H^{s+1}) \cap L^2([0,T), H^{s+1+\frac{\a}{2}}) \\
	\rho& \in C_w([0,T), H^{s+\a}) \cap L^2([0,T), H^{s+1+\frac{\a}{2}}).
	\end{split}
	\end{equation} 
\end{theorem}
Incidentally, local well-posendess in any dimension $n$ can be established too, we refer to \cite{RSlwp} for details.

%%%%%%%%%%%%%%%%%%%%%%%%%%
\subsection{An additional conservation law}\label{s:e}
%%%%%%%%%%%%%%%%%%%%%%%%%
 The  conservative ``$e$''-equation \eqref{e:e} is a heart of matter for the 1D regularity theory, along the lines of \cite{ST1,ST2,ST3,DKRT2018}. We derive it with the use of  the product formula \eqref{e:derL}.
\begin{lemma}[{\bf The conservation law of} $e$] For any solution to the topological model in class \eqref{e:class} the following conservation law holds 
\[
e_t+(ue)_x=0, \qquad e=u_x+\aL_\phi \rho.
\]
\end{lemma}
\begin{proof}
Differentiating the velocity equation and using the product rule \eqref{e:derL} we obtain
\begin{equation}\label{e:uder}
u'_t + u'u' + u u'' = \cL_\phi ((u\rho)') - u' \cL_\phi (\rho) - u (\cL_\phi(\rho))' + \aL_{\phi'}(u\rho).
\end{equation}
The finite difference in the integral representation of the  last term is given by 
\[
u(y)\rho(y) - u(x) \rho(x) = \int_x^y (u\rho)'(\zeta)d\zeta = 
-\int_x^y \rho_t(\zeta) d\zeta= - \p_t\dd_\rho(x,y) \sgn(y-x).
\]
Recalling the formula for the distance $\dd_\rho(x,y) = \left|  \int_x^y  \rho(t,z) \dz \right|$, we obtain
\[
\int_x^y \rho_t(\zeta) d\zeta= \p_t\dd_\rho(x,y) \sgn(y-x).
\]
Thus,
\[
\aL_{\phi'}(u\rho) =  - \int  \p_t\dd_\rho(x,y) \sgn(y-x) \phi'(x,y)\dy.
\]
Noting the relationship
\[
\p_t \dd_\rho(x,y) \sgn(y-x) \phi'(x,y) = \p_t \phi(x,y) (\rho(y) - \rho(x)),
\]
we obtain 
$\displaystyle \aL_{\phi'}(u\rho) = - \int \p_t \phi(x,y) (\rho(y) - \rho(x)) \dy$.
Putting it together with the $\cL_\phi ((u\rho)')$ term we obtain
\[
\cL_\phi ((u\rho)') + \aL_{\phi'}(u\rho) = -\p_t \cL_{\phi}(\rho).
\]
Grouping together terms in \eqref{e:uder} we arrive at 
\[
(u' + \cL_{\phi}(\rho))_t + u' (u' + \cL_{\phi}(\rho)) + u (u' + \cL_{\phi}(\rho))' = 0,
\]
which is precisely the law \eqref{e:e}.
\end{proof}

Paired with the continuity equation we find that the ratio $q = {e}/{\rho}$ satisfies the transport equation
\[
 q_t + u q_x = 0.
\]
Starting from sufficiently smooth initial condition with $\rho_0$ away from vacuum we can assume that $|q(t)|_\infty =  | q_0 |_\infty <\infty$. This gives a priori  pointwise bound
\begin{equation}\label{erho0}
|e(t,x)| \lesssim \rho(t,x).
\end{equation}
The argument can be bootstrapped to higher order derivatives (see \cite[Sec. 2]{ST1}) as follows.  The next order quantity  $q_1= q_x / \rho$ is  again transported
\begin{equation}\label{e:Q}
(q_1)_t + u (q_1)_x = 0.
\end{equation}
Solving for $e'(t,\cdot)$ we obtain another a priori pointwise bound
\begin{equation}\label{erho1}
|e'(t,x)| \lesssim |\rho'(t,x)| + \rho(t,x).
\end{equation}
Continuing in the same manner, $q_2 = (q_1)_x / \rho$, etc,  we obtain
\begin{equation}\label{erhok}
|e^{(k)}(t,x)| \lesssim |\rho^{(k)}(t,x)| +\ldots + \rho(t,x), \qquad k=0,1,2 \ldots
\end{equation}

Using $e$ allows one to rewrite the continuity equation in parabolic form:
\begin{equation}\label{e:rparab}
\rho_t+u\rho_x+e\rho= \rho \aL_\phi(\rho)
\end{equation}
Similarly, one can write the equation for the momentum $m = \rho u$:
\begin{equation}\label{e:m}
m_t+u m_x+e m  = \rho \aL_\phi(m).
\end{equation}
With a priori bounds on the density established in the next section, we can view equations \eqref{e:rparab} -- \eqref{e:m}  as a fractional parabolic system with rough drift and bounded force, which opens a possibility for applying some of the  tools recently developed for such equations.
%%%%%%%%%%%%%%%%%%%%%%%%%
\subsection{Bounds on the density}\label{s:density}
%%%%%%%%%%%%%%%%%%%%%%%%%
Let us first make one trivial remark: if $e_0=0$, then the continuity equation becomes a pure drift-diffusion and hence by the maximum principle the density remains within the confines of its initial bounds:
\begin{equation}\label{e:rhomaxpr}
\rhomin_0 \leq \rho(t,x) \leq \rhomax_0.
\end{equation}
In general, however, the $e$-quantity introduces a Riccati term that needs to be controlled by the singularity of the kernel.  First, we establish a bound from below.

\begin{lemma}\label{l:rhobounds} Let $(\rho,u)$ be a smooth solution of the topological model  \eqref{e:main1D}, with $1\leq \a <2$, subject to initial density $\rho_0$ away from vacuum, $0<\rhomin_0 \leq \rho_0(x) \leq \rhomax_0 <\infty$. 
Then the density obeys the following  bounds  for all  in time:
\begin{equation}\label{e:rhouplow}
 \frac{c}{1+t}   \leq \rho(t,x) \leq \rhomax(\cM_0,|q_0|_\infty,\phi), \qquad  x\in\T, \ t \geq 0,
\end{equation}
\end{lemma}
\begin{proof}
	Let us recall that the continuity equation can be rewritten as
	\begin{equation}\label{e:r-eq}
	\rho_t + u \rho_x = - q \rho^2 + \rho \aL_\phi(\rho).
	\end{equation}
	Let $\rmin$ and $\xmin$ denote the minimum value of $\rho$ and a point where such value is achieved. 	Invoking \lem{l:topoL} to justify the pointwise evaluation we obtain
	\[
	\begin{split}
	\ddt \rmin & \geq - |q_0|_\infty \rmin^2 + \rmin \int_\T  \phi(x_-,y)( \rho(y,t) - \rmin) \dy\geq  - |q_0|_\infty \rmin^2.
	\end{split}
	\]
	The lower bound in \eqref{e:rhouplow} follows.
	
Evaluating the mass equation at extreme maximum we obtain
\[
\ddt \rmax \leq |q_0|_\infty \rmax^2 +  \rmax  \int_{|z| <R_0} \frac{1}{\cM_0|z|^{\a}} ( \rho(t,x_+ + z) - \rmax) \dz.
\]
 Let us further reduce the region of integration to $\e < |z|<R_0$ for any fixed $\e>0$. By choosing $\e$ small enough we can ensure that
\[
 \int_{\e < |z|< R_0 }  \frac{1}{|z|^{\a}} > 2 |q_0|_\infty \cM_0.
\]
Then for that fixed $\e$ we have
\[
\ddt \rmax \leq - |q_0|_\infty \rmax^2  + C \rmax .
\]
The result follows. 
\end{proof}

%%%%%%%%%%%%%%%%%%
\subsection{Continuation of solutions}\label{s:p}
%%%%%%%%%%%%%%%%%
 
Our goal in this section is to establish a general continuation result that relies on 
the uniform H\"older continuity of the density. The latter will be justified in section \ref{sec:Holder}. 

\begin{proposition}\label{p:global}
Consider a local solution to a topological model with $1\leq \a<2$ given by \thm{t:lwp}. 
Suppose there are constants $\rhomin,\rhomax >0$ such that
\begin{equation}\label{e:cC}
\rhomin \leq \rho(t,x) \leq \rhomax, \quad  (t,x)\in   [0,T) \times \T.
\end{equation}
Furthermore, suppose that $\rho$ is uniformly H\"older on $ [0,T)$, i.e., there exists $\gamma>0$ such that
\begin{align}\label{e:modulus}
|\rho(t,x+z) - \rho(t,x)| \leq C|z|^\gamma, \qquad (t,x,z)\in [0,T) \times \T \times \T .
\end{align}
Then the solution remains uniformly in the Sobolev classes $(u,\rho) \in H^{s+1} \times H^{s+\a}$ on $[0,T)$ and, hence, can be continued beyond $T$.
\end{proposition}
\noindent

\begin{proof} 
  We split the proof in five steps. In steps 1-2 we establish control over derivatives of the density up to order $s$. Remarkably this can be done independently of the momentum equation. Such estimates provide bounds on the velocity derivatives up to order $s-1$. On step 3 we develop energy estimates for $\rho^{(s+1)}$ a necessary step before tackling $|u^{(s)}|_\infty$ which is done in step 4. Finally, energy estimates on $u^{(s+1)}$ will finalize the argument with the help of coercivity estimates \eqref{e:coers}. All the pointwise  estimates we used in the proof are presented in Appendix B.

\medskip\noindent 
\steps{1}{Control over $|\rho'|_\infty$} Let us differentiate  \eqref{e:r-eq}:
\begin{equation}\label{e:rhop}
\p_t \rho' + u \rho'' + u' \rho' + e'\rho + e \rho' =  \rho' \cL_\phi \rho + \rho \aL_{\phi'} \rho+ \rho \cL_\phi  \rho' .
\end{equation}
Using again $u' = e - \cL_\phi \rho$ we rewrite
\[
\p_t \rho' + u \rho'' + e'\rho + 2e \rho' =  2\rho' \cL_\phi \rho + \rho \aL_{\phi'} \rho+ \rho \cL_\phi \rho' .
\]
Evaluating at a point $x$ where $|\rho'|$ acheives its maximum and multiplying by $\rho'$ we obtain
\begin{equation}\label{e:rho-dera}
\p_t |\rho'|^2 + e' \rho \rho' + 2 e |\rho'|^2 = 2 |\rho'|^2 \cL_\phi \rho + \rho' \rho \aL_{\phi'} \rho +  \rho \rho' \cL_\phi \rho'.
\end{equation}
In view of \eqref{erho0} and \eqref{erho1} we can bound
\[
|e' \rho \rho' + 2 e |\rho'|^2| \leq C (|\rho'|^2 + |\rho'|).
\]
Thus,
\begin{equation}\label{e:rho-dera2}
\p_t |\rho'|^2 \leq  C (|\rho'|^2 + |\rho'|) + 2 |\rho'|^2 \cL_\phi \rho  + \rho' \rho \aL_{\phi'} \rho +  \rho \rho' \cL_\phi \rho'.
\end{equation}
Let us note in passing that \lem{l:topoL} justifies pointwise evaluation of all operators involved.  Due to the bound from below on $\rho$, the last term provides  dissipation. Indeed, let us note the identity
\[
\rho'(x) \d_z \rho'(x) = - \frac12 | \d_z \rho'(x)|^2 + \frac12( (\rho'(x+z))^2 - (\rho'(x))^2).
\]
Since $x$ is a point of maximum, we can see that the second difference is negative. Thus,
\begin{equation}\label{e:disslower}
\rho \rho'  \cL_\phi \rho' \leq  - c_1 \dD_\a \rho'(x).
\end{equation}
where 
\[
\dD_\a\rho'(x) =  \int_\R \frac{|\d_z \rho'(x)|^2}{|z|^{1+\a}} h(z) \dz.
\]
Let us recall the nonlinear estimate on $\dD_\a\rho'(x) $ obtained in Constantin and Vicol \cite{CV2012}, which will play a crucial role in what follows:
\begin{equation}\label{e:nlmp}
	\dD_\a\rho'(x) \geq C \frac{|\rho'(x)|^{2+\a}}{|\rho|_\infty^\a} - c |\rho'|^2_2.
\end{equation}
Here the $- c |\rho'|^2_2$ appears when we complement the cutoff function $h$ to the full unity. Given a uniform bound on $|\rho|_\infty$ we further estimates, keeping half of the dissipation as is for subsequent usage,
\begin{equation}\label{e:D12}
\dD_\a\rho'(x) \geq \frac12 \dD_\a\rho'(x) + C |\rho'(x)|^{2+\a} - c |\rho'|^2_2 .
\end{equation}
Because of the second term in \eqref{e:D12}, all powers of $\rho$ below $2+\a$ which appear in \eqref{e:rho-dera2} are absorbed. So, in particular at this stage we can rewrite  \eqref{e:rho-dera2} as 
\begin{equation}\label{e:rho-dera2a}
\p_t |\rho'|^2 \leq C + 2 |\rho'|^2 \cL_\phi \rho + \rho' \rho \aL_{\phi'} \rho - \frac12 \dD_\a\rho'(x) - c |\rho'(x)|^{2+\a}.
\end{equation}
We now invoke the estimates on the operators $\cL_\phi \rho$ and $ \aL_{\phi'} \rho$ obtained in \lem{l:Linfty} and \lem{l:1stcomm}, respectively.   We have, knowing that by our assumption the density is uniformly H\"older continuous ,
\begin{equation}\label{e:intermed1}
	|\rho'|_\infty^2 |\cL_\phi \rho (x)| \lesssim  r^{1 - \frac{\a}{2}} |\rho'|_\infty^2 \sqrt{\dD_\a[ \rho'] (x)} + r^{\g - \a} |\rho'|_\infty^2 + r^{2-\a}  |\rho'|^4_\infty.
\end{equation}
Let us fix a small $\e>0$ to be determined later, and define $r = \frac{\e}{ |\rho'|_\infty}$. Then the above is bounded by 
\[
\begin{split}
& \lesssim  \e^{1 - \frac{\a}{2}} |\rho'|_\infty^{1 + \frac{\a}{2}} \sqrt{\dD_\a[ \rho'] (x)} + c_\e |\rho'|_\infty^{2+\a - \g} + \e^{2-\a}  |\rho'|^{2+\a}_\infty \\
& \lesssim  \e^{1 - \frac{\a}{2}} |\rho'|_\infty^{2 + \a} +  \e^{1 - \frac{\a}{2}}  \dD_\a[ \rho'] (x) + \e |\rho'|_\infty^{2+\a} + C_\e + \e^{2-\a}  |\rho'|^{2+\a}_\infty.
\end{split}
\]
For $\e$ sufficiently small we can see that all these terms except for the free constant get absorbed into the dissipation term in view of \eqref{e:D12}. Continuing to the next term, in view of \lem{l:1stcomm}, and with the same choice of scale $r$, we obtain
\[
|\rho'|_\infty |\aL_{\phi'} \rho(x)| \lesssim  r^{1 - \frac{\a}{2}} |\rho'|^2_\infty \sqrt{\dD_\a[ \rho'] (x)}+ r^{\g - \a}  |\rho'|^2_\infty + r^{2-\a} |\rho'|^4_\infty,
\]
which exactly repeats \eqref{e:intermed1}. 

Going back to \eqref{e:rho-dera2a} we arrive at
\begin{equation}\label{e:rho-dera5}
\p_t |\rho'|^2 \leq   c_1 - c_2 \dD_\a \rho'.
\end{equation}
This finishes the proof of control over $\rho'$.

\medskip\noindent 
\steps{2}{Control over $|\rho^{(s)}|_\infty$ and $|u^{(s-1)}|_\infty$}  We now establish uniform control over the maximal allowed derivative of $\rho$ in $L^\infty$ metric. Note that $H^{s + \a}$ embeds into $W^{s,\infty}$ for the range in question $1\leq \a <2$. So, initially and on the local time interval $[0,T)$ we have the density in $W^{s,\infty}$ class non-uniformly at the moment. Once this step is accomplished we obtain automatically a uniform bound on $u^{(s-1)}$. Indeed, by \lem{l:Lkbound},
\begin{multline*}
|u^{(s-1)}|_\infty\leq |e^{(s-2)}|_\infty+|(\cL_\phi \rho)^{(s-2)}|_\infty \lesssim  |\rho^{(s-2)}|_\infty + \sqrt{\dD_\a[\rho^{(s-1)}]} +  |\rho^{(s-1)}|_\infty 
\lesssim  C+ |\rho^{(s)}|_\infty.
\end{multline*}

 We will argue by induction. The initial hypothesis was established on the previous step. Let us now assume that we have a uniform control over $|\rho^{(k-1)}|_\infty$ for $2 \leq k\leq s$, and obtain control over $|\rho^{(k)}|_\infty$.  

Differentiating the continuity equation $k$ times and expanding we obtain
\begin{equation}\label{e:rhok}
	\p_t \rho^{(k)} + u \rho^{(k+1)} + \sum_{l=1}^k c_{k,l} u^{(l)} \rho^{(k+1 - l)} +  u^{(k+1)} \rho =0.
\end{equation}
Evaluating at the maximum of $|\rho^{(k)}|$ and multiplying by $\rho^{(k)}$ the term $u \rho^{(k+1)}$  drops out. In the rest we replace all $u$'s with the corresponding $e$-expression. So, let us consider the end-point case first,
\[
 \rho^{(k)} u^{(k+1)} \rho =  \rho^{(k)} e^{(k)} \rho - \rho^{(k)} (\cL_\phi \rho)^{(k)} \rho .
 \]
By the induction hypothesis and \eqref{erhok} we have
\[
|e^{(k)}| \lesssim |\rho^{(k)}| + C,
\]
and so,
\[
| \rho^{(k)} e^{(k)} \rho | \lesssim  |\rho^{(k)}|^2 + C.
\]
Next, we have 
\[
 (\cL_\phi \rho)^{(k)} = \cL_\phi (\rho^{(k)}) +  [(\cL_\phi \rho)^{(k)}-  \cL_\phi (\rho^{(k)})].
\]
For the dissipation, we have as usual, in view of the nonlinear maximum estimate and the fact that $\rho^{(k-1)}$ is under control:
\begin{equation}\label{e:dissrhok}
 \rho^{(k)} \cL_\phi (\rho^{(k)}) \lesssim - \dD_\a [\rho^{(k)}](x) -  |\rho^{(k)}|_\infty^{2+\a}.
\end{equation}
For the commutator we encounter a cubic term of top order in the case $k=2$. Therefore we use \eqref{e:k2} with small $\e$ 
\[
	|\rho'' [(\cL_\phi \rho)'' -  \cL_\phi (\rho'')]| \lesssim |\rho''|_\infty \sqrt{ \dD_\a [\rho''](x)} + \e |\rho''|_\infty^3 +C_\e \lesssim |\rho''|_\infty^2 + \e \dD_\a [\rho''](x)+\e |\rho''|_\infty^3 +C_\e.
\]
In view of \eqref{e:dissrhok}, the terms $\e \dD_\a [\rho''](x)+\e |\rho''|_\infty^3$ are absorbed by dissipation. For general $k \geq 3$ , we use \eqref{e:k3} by replacing $ \sqrt{\dD_\a[ \rho^{(k-1)}] (x)} \lesssim  | \rho^{(k)}|_\infty$:
\[
	| \rho^{(k)} (\cL_\phi f)^{(k)} -  \cL_\phi (f^{(k)}) (x) |  \lesssim 	| \rho^{(k)}|_\infty \sqrt{\dD_\a[ \rho^{(k)}] (x)} 
	+| \rho^{(k)}|_\infty^2  + 1 \lesssim \e \dD_\a[ \rho^{(k)}] (x) + c_\e | \rho^{(k)}|_\infty^2  +1.
\]
Again, the dissipation term is absorbed.

Next, let us look into intermediate terms, $1\leq l \leq k$,
\[
\rho^{(k)} u^{(l)} \rho^{(k+1 - l)} = \rho^{(k)} e^{(l-1)} \rho^{(k+1 - l)} - \rho^{(k)} (\cL_\phi \rho)^{(l-1)} \rho^{(k+1 - l)} .
\]
Since $l- 1 \leq k-1$ we have all $e^{(l-1)} $ uniformly bounded, hence, 
\[
|\rho^{(k)} e^{(l-1)} \rho^{(k+1 - l)} | \lesssim |\rho^{(k)} |^2 + C.
\]
Finally for the remaining terms $\rho^{(k)} (\cL_\phi \rho)^{(l-1)} \rho^{(k+1 - l)}$ we appeal to \lem{l:Lkbound}. So, if  $l=k$,  by \eqref{e:LkboundA}
\[
| \rho^{(k)} (\cL_\phi \rho)^{(k-1)}(x) \rho' | \lesssim |\rho^{(k)} |_\infty  ( \sqrt{\dD_\a[ \rho^{(k)}] (x)} + |\rho^{(k)}|_\infty+1) \lesssim \e \dD_\a[ \rho^{(k)}] (x) + c_\e |\rho^{(k)} |_\infty^2 + 1,
\]
so this term is taken care of.  For $l=k-1$, if $k=2$, we estimate using more refined bound \eqref{e:Linfty},
\[
| \rho''  (\cL_\phi \rho) \rho'' | \lesssim  \e | \rho''|_\infty^3 + c_\e  | \rho''|_\infty^2,
\]
which is absorbed. And for $k>2$, we obtain from \eqref{e:LkboundB}
\[
| \rho^{(k)} (\cL_\phi \rho)^{(k-2)} \rho'' | \lesssim | \rho^{(k)}|_\infty^2 + 1.
\]
Finally, for all $1\leq l\leq k-2$, we use \eqref{e:LkboundC}
\[
| \rho^{(k)} (\cL_\phi \rho)^{(l-1)} \rho^{(k+1-l)} | \lesssim  | \rho^{(k)}|_\infty^2 + 1.
\]
We have obtained
\[
\p_t  |\rho^{(k)} |^2_\infty \lesssim |\rho^{(k)} |^2_\infty+1,
\]
and the result follows. 

\medskip\noindent 
\steps{3}{Energy estimates for  $\rho^{(s+1)}$} Before going into estimates for the momentum, we make one more intermediate step by establishing that $\rho^{(s+1)} \in L^\infty_t L^2_x \cap L^2_t H^{\a/2}_x$. The basic energy estimate for $\rho^{(s+1)}$ is obtained in the standard way. To simplify some computations, let us note  the a priori bound 
\[
\| u \|_{C^{s-1}} \leq \| e \|_{C^{s-2}} +\| \cL_\phi \rho \|_{C^{s-2}} \lesssim \| \rho \|_{C^{s-2}} + \sqrt{\dD_\a[\rho^{(s-1)}]} + \| \rho \|_{C^{s-1}} 
\lesssim  C+ \| \rho \|_{C^{s}}  \leq C.
\]
With this and expansion \eqref{e:rhok} we obtain
\[
\ddt | \rho^{(s+1)} |_2^2 \lesssim  | \rho^{(s+1)} |_2^2 + \int_\T \rho^{(s+1)} u^{(s)} \rho'' \dx + \int_\T \rho^{(s+1)} u^{(s+1)} \rho' \dx +\int_\T \rho^{(s+1)} u^{(s+2)} \rho \dx .
\]
By replacing the remaining velocities with $e - \cL_\phi \rho$ we now estimate each term:
\[
\left| \int_\T \rho^{(s+1)} u^{(s)} \rho'' \dx \right| \lesssim |\rho^{(s+1)} |_1 |\rho^{(s-1)}|_\infty  |\rho''|_\infty + \left| \int_\T \rho^{(s+1)} (\cL_\phi \rho)^{(s-1)} \rho'' \dx \right|
\]
applying \eqref{e:LkboundB} with $k = s+1$,
\[
\lesssim |\rho^{(s+1)} |_2 + |\rho^{(s+1)} |_2^2 + \| \rho\|^2_{H^{s+\a/2}} \lesssim |\rho^{(s+1)} |_2^2+\e\| \rho\|^2_{H^{s+1+\a/2}} + c_\e.
\]
The $H^{s+1+\a/2}$-norm  will be absorbed into dissipation. Next,
 \[
 \int_\T \rho^{(s+1)} u^{(s+1)} \rho' \dx \lesssim \int_\T \rho^{(s+1)} \rho^{(s)} \rho' \dx + \int_\T \rho^{(s+1)} (\cL_\phi \rho)^{(s)}\rho' \dx,
 \]
 and applying \eqref{e:LkboundA}, 
 \[
\lesssim  |\rho^{(s+1)} |_2^2 + \int_\T |\rho^{(s+1)}(x)|\sqrt{\dD_\a[\rho^{(s+1)}](x)} \dx + \| \rho\|^2_{H^{s+\a/2}}  \lesssim |\rho^{(s+1)} |_2^2+\e\| \rho\|^2_{H^{s+1+\a/2}} + c_\e.
\]
Finally,
\[
\int_\T \rho^{(s+1)} u^{(s+2)} \rho \dx  = \int_\T \rho^{(s+1)} e^{(s+1)} \rho \dx  -\int_\T \rho^{(s+1)} (\cL_\phi \rho)^{(s+1)}\rho \dx .
\]
Via \eqref{erhok} the first term is bounded by $|\rho^{(s+1)} |_2^2$. As for the second we use commutator estimates
\begin{multline*}
\int_\T \rho^{(s+1)} (\cL_\phi \rho)^{(s+1)}\rho \dx \lesssim  - \| \rho\|^2_{H^{s+1+\a/2}} + \int_\T| \rho^{(s+1)}(x)|\sqrt{\dD_\a[\rho^{(s+1)}](x)}\dx\\
+\int_\T| \rho^{(s+1)}(x)|\sqrt{\dD_\a[\rho^{(s)}](x)}\dx
+ |\rho^{(s+1)} |_2^2 \lesssim  - \| \rho\|^2_{H^{s+1+\a/2}} + \e\| \rho\|^2_{H^{s+1+\a/2}} + c_\e|\rho^{(s+1)} |_2^2.
\end{multline*}
All the estimates now add up to 
\[
\ddt | \rho^{(s+1)} |_2^2 \lesssim -\frac12  \| \rho\|^2_{H^{s+1+\a/2}} + c_\e | \rho^{(s+1)} |_2^2 + c_\e.
\]
This shows $\rho^{(s+1)} \in L^\infty_t L^2_x \cap L^2_t H^{\a/2}_x$ and the step is complete.

\medskip\noindent 
\steps{4}{Control over $|u^{(s)}|_\infty$}   Due to close resemblance of the momentum equation \eqref{e:m} to the continuity equation written in parabolic form \eqref{e:rparab} it is easier to work with the momentum variable $m$. Since all the high order spaces are Banach algebras, establishing control over $m$ is equivalnent to establishing control over $u$:
\[
\|u \|_X \lesssim \|m\|_X \left\| \rho^{-1} \right\|_X \lesssim \|m\|_X \|\rho\|_X, \quad \|m\|_X \lesssim \|u\|_X \|\rho\|_X,
\]
which applies to $X = H^s, C^s$, etc. Knowing that $\rho \in X$ shows $\|u\|_X \sim \|m\|_X$. In particular this is the case for all $C^k$, $k\leq s$.

We do have automatic uniform bound in $C^{s-1}$ as a consequence of the previous step. Indeed, by \lem{l:Lkbound},
\begin{multline*}
	\| m \|_{C^{s-1}} \lesssim \| u \|_{C^{s-1}} \leq \| e \|_{C^{s-2}} +\| \cL_\phi \rho \|_{C^{s-2}} \lesssim \| \rho \|_{C^{s-2}} + \sqrt{\dD_\a[\rho^{(s-1)}]} + \| \rho \|_{C^{s-1}} \\
	 \lesssim  C+ \| \rho \|_{C^{s}}  \leq C.
\end{multline*}
So, essentially we need to complete one more step up. 

Differentiating  \eqref{e:m} $s$ times, testing with $m^{(s)}$ and evaluating at the maximum, we obtain
\[
\p_t | m^{(s)} |_\infty^2 + m^{(s)}\sum_{l=1}^s u^{(l)} m^{(s+1-l)}  + (em)^{(s)} m^{(s)} = (\rho \cL_\phi m)^{(s)} m^{(s)}.
\]
The $e$-term is under control:
\[
| (em)^{(s)} m^{(s)} | \leq  |(em)^{(s)}|_\infty |m^{(s)} |_\infty \lesssim \|\rho\|_{C^s} \|m\|^2_{C^s} \lesssim \|m\|^2_{C^s}.
\]
Next, using the induction hypothesis,
\[
\left|m^{(s)} \sum_{l=1}^s u^{(l)} m^{(s+1-l)}  \right| \lesssim (|u'|+...+|u^{(s-1)}|)  |m^{(s)} |^2_\infty +   |u^{(s)}| |m^{(s)} | |m'|_\infty \lesssim  |m^{(s)} |^2_\infty.
\]
So, further argument is reduced to estimating the dissipation term. We have for all $1 \leq l \leq s$,
\[
|m^{(s)} \rho^{(l)} (\cL_\phi m)^{(s-l)} (x) | \lesssim |m^{(s)}|_\infty | (\cL_\phi m)^{(s-l)}(x) |
\]
and using \lem{l:Lkbound},
\[
\lesssim  |m^{(s)}|_\infty \left( \sqrt{\dD_\a[m^{(s)}](x)} +   |m^{(s)}|_\infty +C \right) \lesssim c_\e  |m^{(s)}|_\infty^2 + \e \dD_\a[m^{(s)}](x) + C.
\]
The $D_\a$-term will be absorbed subsequently.  So, it comes down to 
\[
\rho (\cL_\phi m)^{(s)} m^{(s)}.
\]
As usual, $\cL_\phi (m^{(s)}) m^{(s)}$ produces dissipation $\dD_\a[m^{(s)}](x)$, and all that remains  to estimate is the commutator, for which we use \lem{l:kcomm} with $r \sim 1$,
\begin{multline*}
	|m^{(s)}|_\infty	| (\cL_\phi m)^{(s)} (x) -  \cL_\phi (m^{(s)}) (x) |   \lesssim    |m^{(s)}|_\infty \left(\sqrt{\dD_\a[ m^{(s)}] (x)} +  \sqrt{\dD_\a[ \rho^{(s)}] (x)} \right) + |m^{(s)}|^2_\infty  \\
	c_\e  |m^{(s)}|_\infty^2 + \e \dD_\a[ m^{(s)}] (x) + \e \dD_\a[ \rho^{(s)}] (x).
\end{multline*}
It remains to notice that 
\[
|\dD_\a[ \rho^{(s)}] (x) | \lesssim |\rho^{(s+1)}|_q^2, \quad \text{for } q > \frac{2}{2-\a}.
\]
Since $H^{\a/2} \hookrightarrow H^{1/2} \hookrightarrow L^q$, for any $q<\infty$, we have $|\rho^{(s+1)}|_q^2\in L^1$ by the previous step. We conclude that the term $\dD_\a[ \rho^{(s)}] (x)$ is $L^1$-integrable in time. Thus,
\[
\p_t | m^{(s)} |_\infty^2 \lesssim | m^{(s)} |_\infty^2 - \dD_\a[ m^{(s)}]  + f(t),
\]
where $f\in L^1([0,T])$. This finishes the step.

\medskip\noindent 
\steps{5}{Energy estimates for  $u^{(s+1)}$ and conclusion of the proof}  Since the momentum equation is structurally the same as the continuity, this step is entirely similar to Step 4. The use of commutator estimates of \lem{l:kcomm} and \lem{l:Lkbound} is identical with $f = m$ due to the fact that at this point we are in the same position in terms of control of $m$ as we were at the beginning of Step 4. We thus conclude
\[
m^{(s+1)} \in L^\infty_t L^2_x \cap L^2_t H^{\a/2}_x,
\]
and via Banach algebra inequality $\|u\|_X \leq \|m\|_X \|1/\rho\|_X\sim \|m\|_X \|\rho\|_X$ fo the classes in question, we obtain 
\[
u^{(s+1)} \in L^\infty_t L^2_x \cap L^2_t H^{\a/2}_x.
\]
To conclude the proof it remains to notice that via the $e$-quantity, we have $(\cL_\phi \rho)^{(s+1)} \in L^\infty_t L^2_x$. Due to \eqref{e:coers},
\[
\|\rho\|^2_{{H}^{s+\alpha}} \lesssim C+ \| \rho\|_{{H}^{s+\a/2}}^N \lesssim C + \| \rho\|_{{H}^{s+1}}^N.
\]
On Step 3 we already established uniform control over $\| \rho\|_{{H}^{s+1}}$.  The proof is finished. 
\end{proof}  

\ifx%%
Let us note that the more natural space for the density-enstrophy would be the class $L^2_t H_x^{s+\a +\frac{\a}{2}}$ for $\a>1$. While it is very likely that the density  does belong to this class, proving this would involve rather technical fractional energy estimates directly in $H^{s+\a}$, which we will postpone to future work. 
\fi%%%
 
 %%%%%%%%%%%%%%%%%%%%%%%%%%%%%%%%%
\subsection{H\"older regularization of the density}\label{sec:Holder}

%%%%%%%%%%%%%%%%%%%%%%%%%%%%%%%%
In this section we  derive the H\"older regularity of the density ---  its H\"older \emph{regularization} follows from the fractional diffusion embedded in our topological alignment term. The proof is obtained by various techniques of fractional  parabolicity depending on $\a$. Combined with \prop{p:global}, we immediately obtain global existence and conclude \thm{t:mainI}.

\subsubsection{{\bf Case $1<\a<2$ via Schauder}.}\label{s:sch}

In this particular case the regularization will follow from a kinematic argument based on the Schauder estimates as in \cite{CS2011,JX2015}.  So, we start by rewriting the relation between $\rho$, $u$, and $e$ as follows
\begin{equation}\label{e:Sch1}
\p_x^{-1}  \cL_\phi \rho = \p_x^{-1} e - u \in L^\infty.
\end{equation}
In the purely metric case this of course implies $\rho\in C^{1-\a}$ immediately. For the topological models the conclusion is not so straightforward,  and in fact may not even be true up to regularity $1-\a$.  

First let us make an observation that $ \cL_\phi \rho = \p_x (\cF \rho)$, where 
\[
\cF \rho(x) =  \int \frac{\sgn(z) \ln \dd_\rho(x+z,x) }{ |z|^{\a}} h(z) \dz.
\]
Next, by symmetrization 
\[
\cF \rho(x) = \frac12 \int \frac{ \ln \dd_\rho(x+z,x) - \ln \dd_\rho(x-z,x)  }{ |z|^{\a}} \sgn(z) h(z) \dz.
\]
Now we use the expansion 
\begin{equation}\label{e:taylorlog}
	\begin{split}
&	\ln \dd_\rho(x+z,x) -  \ln \dd_\rho(x-z,x) \\
	& =[\dd_\rho(x+z,x) -   \dd_\rho(x-z,x)] \int_{0}^1 \frac{\dth}{ \th \dd_\rho(x+z,x)+ (1-\th) \dd_\rho(x-z,x)}.
	\end{split}
\end{equation}
Next,    
\[
[\dd_\rho(x+z,x) - \dd_\rho(x-z,x)] \sgn(z) = \int_x^{x+z} \rho(y) \dy + \int_x^{x-z} \rho(y) \dy = \int_{-z}^z \rho(x+w) \sgn w \dw.
\]
We can now subtract the total mass from the density without changing the result. However, the function $\rho - \cM_0$ is a mean-zero function.  Hence, $\rho - \cM_0 = f'$, for some $f$. Continuing we obtain
\[
[\dd_\rho(x+z,x) - \dd_\rho(x-z,x)] \sgn(z) = \int_{-z}^z f'(x+w) \sgn(w)\dw = f(x+z)+f(x-z) - 2 f(x),
\]
which is the second order finite difference of $f$.  We thus obtain
\[
\cF \rho(x) = \int  [f(x+z)+f(x-z) - 2 f(x)] K(x,z,t) \dz,
\]
where the kernel $K(x,z,t)$ is given  by
\[
K(x,z,t) =  \frac{h(z)}{|z|^{\a}}  \int_{0}^1 \frac{\dth}{ \th \dd_\rho(x+z,x)+ (1-\th) \dd_\rho(x-z,x)}.
\]
It satisfies the following four conditions:
\begin{itemize}
    \item[(i)] $\displaystyle{  \frac{\one_{|z|<R_0}}{|z|^{1+\a}}\lesssim K(x,z,t) \lesssim  \frac{\one_{|z|<2R_0}}{|z|^{1+\a}}}$;
    \item[(ii)] $K(x,-z,t) = K(x,z,t)$;
    \item[(iii)]  $\displaystyle{|z|^{2+\a}| K(x+h,z,t) - K(x,z,t)| \leq C |h|}$;
\item[(iv)]  $\displaystyle{|\p_z(|z|^{1+\a} K(x,z,t))| \leq C |z|^{-1}}$.    
\end{itemize}
Here the inequalities involve  generic  constants which may depend only on the density but  not on its derivatives.  Indeed, (i) is trivial. As to (iv), we have
\begin{equation}\label{e:kiv}
|z|^{1+\a} K(x,z,t) = h(z) |z| \int_0^1 \left[ \th \dd_\rho(x+z,x) +(1-\th)\dd_\rho(x-z,x) \right]^{-1} \dth.
\end{equation}
Given that $\dd_\rho(x+z,x) \sim |z|$, it is clear that this expression is uniformly bounded by a constant.  It will remain so if $\p_z$ falls on $h$. The bound gains a negative power $|z|^{-1}$ when $\p_z$ falls on $|z|$. Next, observe that
\[
\p_z \dd_\rho(x\pm z,x) = \rho(x \pm z) \sgn(z),
\]
which is a uniformly bounded quantity. So, any derivative that falls on the distance inside the expression \eqref{e:kiv} reduces the power of that term by $1$, while the rest remains uniformly bounded. 

To verify (iii) we can even prove a stronger inequality
\[
|z|^{2+\a}| \p_x K(x,z,t)| \leq C .
\]
Indeed, in this case we recall \eqref{e:derd} which implies that $\p_x \dd_\rho(x \pm z,x)$ remains uniformly bounded. So, we have
\[
|z|^{2+\a}\p_x K(x,z,t) = h(z)|z|^{2} \p_x   \int_0^1 \left[ \th \dd_\rho(x+z,x) +(1-\th)\dd_\rho(x-z,x) \right]^{-1} \dth.
\]
In view of the above observation, the order of the partial of the entire expression in parenthesis is $|z|^{-2}$.  This finishes the verification.

So, the initial relation \eqref{e:Sch1} can be stated now as a fractional elliptic problem:
\begin{equation}\label{e:Sch2}
\int  [f(x+z)+f(x-z) - 2 f(x)] K(x,z,t) \dz = g(x)  \in L^\infty.
\end{equation}
Under the assumptions (i) -- (iv), it is known, see for example \cite{CS2011,JX2015}, that any bounded solution $f$ to \eqref{e:Sch2} satisfies $f \in  C^{1+\g}$ for some positive $\g>0$. This readily implies $\rho \in C^{\g}$ and concludes the argument.

%%%%%%%%%%%%%%%%%%%%%%%%%%%%%%
\subsubsection{{\bf Case  $\a=1$ via De Giorgi}.}\label{s:degiorgi}
%%%%%%%%%%%%%%%%%%%%%%%%%%%%%%%
In this section we present a regularization result for the case $\a =1$. We state our result more precisely in the following proposition.

\begin{proposition}\label{p:degiorgi} Consider the case  $\a = 1$. Assume the  density is uniformly bounded \eqref{e:cC}. Then there exists a $\g >0$ such that $\displaystyle [\rho]_{\g} \leq \frac{C}{t^\g}$ for all $t\in (0,T]$. Here $C$ depends on the bounds on the density on $[0,T]$.
\end{proposition}

Let us make some preliminary remarks. Our proof is based on blending our model into the settings of Caffarelli, Chan, Vasseur work \cite{CCV2011} which adopts the method of De  Giorgi to non-local equation with symmetric kernels. We note however that the result of \cite{CCV2011} is not directly applicable to our model due to the presence of drift and force in the continuity equation, and in addition we  lack symmetry of the kernel. The forced case was considered in a similar situation in Golse et al \cite{GIMV2017}, where the control over the force is achieved via pre-scaling of the equation. We will use a similar argumentation here as well.  We proceed in five steps.

\medskip
\noindent  
\steps{1}{Symmetric form of the continuity equation}  Let us recall the continuity equation in parabolic form:
\begin{equation}\label{e:rhoe0}
    \rho_t + u \rho_x = \rho \cL_\phi \rho - e\rho.
\end{equation}
To get rid of the $\rho$ prefactor we will perform the following procedure: divide \eqref{e:rhoe0} by $\rho$ and write evolution equation for the new variable $w = \ln \rho$,
\[
w_t + u w_x = \cL_\phi e^w  - e.
\]
Using that
\[
e^{w(y)} - e^{w(x)} = (w(y) - w(x)) \int_0^1 \rho^\th(y) \rho^{1-\th}(x) \dth,
\]
we further rewrite the equation as
\begin{equation}\label{e:w}
w_t + u w_x = \cL_K w  - e.
\end{equation}
where 
\[
K(x,y,t) = \phi(x,y) \int_0^1 \rho^\th(y) \rho^{1-\th}(x)\dth 
\]
In view of the bounds on the density, the new kernel satisfies 
\begin{equation}\label{e:kerK}
\frac{\one_{|x-y|<R_0}}{|x-y|^{1+\a}} \lesssim K(x,y) \lesssim \frac{\one_{|x-y|<2R_0}}{|x-y|^{1+\a}},
\end{equation}
and now is fully symmetric
\[
K(x,y,t) = K(y,x,t).
\]
Clearly,   H\"older continuity of $w$ is equivalent to that of $\rho$, so we will work with \eqref{e:w} instead.

In what follows we treat the term $-e$ as a passive source. However we cannot treat $u$ similarly since the derivative $u_x$ that will come up in the truncated energy inequality will have to be recycled back through its connection with $e$. We therefore first discuss scaling properties of the system. 

\medskip
\noindent  
\steps{2}{Rescaling} Let us adopt the point of view that our solution $(u,\rho)$ is defined periodically on the real line $\R$. Elementary computation shows that if 
$(u,\rho)$ is a solution and $R > 0$, then the new pair 
\begin{equation}\label{e:urresc}
u_R = u\left( t_0 + \frac{t}{R^\a}, x_0 + \frac{x}{R}\right) , \quad \rho_R = \rho\left( t_0 + \frac{t}{R^\a}, x_0 + \frac{x}{R}\right) 
\end{equation}
satisfies the rescaled system

\begin{equation}\label{e:rescale}
\left\{
\begin{split}
\p_t\rho_R +R^{1-\a}(\rho_R u_R)_x&=0,\\
\p_t u_R +R^{1-\a} u_R u_R' &=\int_\R \rho_R(y) (u_R(y) - u_R(x)) \phi_R(x,y) \dy,
\end{split}\right.
\end{equation}
where the new kernel is given by
\[
\phi_R(x,y,t) = \frac{1}{R^{1+\a}} \phi\left( x_0 + \frac{x}{R},x_0 + \frac{y}{R}, t_0 + \frac{t}{R^\a} \right).
\]
Note that for a given bound on the density $c<\rho<C$ on a given time interval $I$, the new kernel still satisfies
\[
\l \frac{\one_{|x-y| \leq R_0 R}}{|x-y|^{1+\a}}  \leq \phi_R(x,y)  \leq \L \frac{\one_{|x-y|<2R_0 R}}{|x-y|^{1+\a}},
\]
on time interval $R^{\a}(I - t_0)$, and the constants $\L,\l$ are independent of $R$. Thus, if $R>1$, the bound from below holds on a wider space and time intervals. The corresponding $e$-quantity rescales to
\[
e_R = R^{1-\a} u_R' + \cL_{\phi_R} \rho_R = \frac{1}{R^\a} e\left( t_0 + \frac{t}{R^\a}, x_0 + \frac{x}{R}\right) ,
\]
and satisfies 
\[
\p_t e_R + R^{1-\a}( u_R e_R)_x = 0.
\]
Hence, $e_R / \rho_R$ is transported and as a consequence we obtain an priori bound 
\begin{equation}\label{e:eR}
|e_R|\lesssim \frac{1}{R^\a} \rho_R\lesssim \frac{1}{R^\a}.
\end{equation}
The rescaled continuity equation becomes
\[
\p_t \rho_R + R^{1-\a} u_R \rho'_R + e_R \rho_R = \rho_R \cL_{\phi_R} \rho_R.
\]
The corresponding $w$-equation reads
\[
\p_t w_R  + R^{1-\a} u_R w_R' = \cL_{K_R} w  - e_R,
\]
where the kernel $K_R$ satisfies the same bound \eqref{e:kerK} for all $R\geq 1$. 

So, it is clear that the drift remains under control for $\a \geq 1$, and is scaling invariant in the case $\a=1$.

\medskip
\noindent  
\steps{3}{First De Giorgi lemma}  We return to the symmetrized version of the continuity equation \eqref{e:w}, where the only extra term the prevents us to directly apply \cite{CCV2011} is the drift. Since, in addition the drift is not div-free and non-linearly depends upon $\rho$ we will take extra care of keeping protocol of relation between  $w$ and $u$ after re-scalings.  

First, we start by noting that it suffices to work on time interval $[-3,0]$ and  prove uniform H\"older continuity on  $[-1,0]$. Second, in view of \eqref{e:eR} if necessary we can rescale the equation by a large $R>1$ and assume without loss of generality that $| e|_{L^\infty(\R \times [-3,0))} = \e_0 <1$, where $\e_0$ will be determined at a later stage and will in fact depend only on $\L,\l$.

The argument of \cite{CCV2011} uses rescaling of the form $\w = \frac{w_R }{C_1} + C_2$, where $R\geq 1$, and $0< C_1 \leq C_0=\max\{ 1, |w|_\infty\}$, and $w$ is the original solution, and $C_2$ is a constant which changes from step to step. Let us note that the new quantity $\w$ satisfies 
\begin{equation}\label{e:wR}
\begin{split}
\w_t + u_R \w_x &= \cL_{K_R} \w + f_{R,C_1},\\
|f_{R,C_1}|_\infty & \leq \frac{\e_0}{R C_1}.
\end{split}
\end{equation}
To keep control over the source we therefore impose the following assumption on all rescalings
\begin{equation}\label{e:RC}
RC_1 >1.
\end{equation}
We will now derive a truncated energy inequality for $\w$.

Let $\psi$ be a Lipschitz function on $\R$. We always assume that our Lipschitz functions have slopes bounded by a universal constant. Testing \eqref{e:wR} with $(\w-\psi)_+$ we obtain
\[
\begin{split}
\frac12 \ddt \int_\R (\w-\psi)_+^2 \dx  & - \frac12 \int (u_R)_x (\w-\psi)_+^2 \dx - \frac12 \int u_R \psi_x (\w-\psi)_+ \dx\\
& = - B_R(\w,(\w-\psi)_+) + \int f_{R,C_1} (\w-\psi)_+ \dx,
\end{split}
\]
where
\[
{B_R}(h,g) = \frac12 \int K_R(x,y) (h(y) - h(x)) (g(y) - g(x)) \dy \dx.
\]
Continuing we obtain
\[
(u_R)_x = e_R - \cL_{\phi_R} \rho_R = e_R -\cL_{K_R} w_R = e_R -C_1 \cL_{K_R} \w. 
\]
We also note that in view of our assumptions and the maximum principle we have a  scaling invariant bound
$|u_R \psi_x| \leq C$. So, as long as in addition $RC_1>1$, we obtain
\[
\frac12 \ddt \int_\R (\w-\psi)_+^2 \dx + B_R(\w,(\w-\psi)_+) \leq \frac{C_1}{2} B_R(\w,(\w-\psi)_+^2) + C (|(\w-\psi)_+|_1 + |(\w-\psi)_+|_2^2).
\]
Note that the $B$-term on the right hand side is cubic, while on the left hand side it is quadratic. This will help hide the cubic term with the help of  the following smallness assumption: 
\begin{equation}\label{e:e0}
|(\w-\psi)_+|_\infty \leq \frac{1}{2C_0}.
\end{equation}
Under this assumption we have
\[
B_R(\w,(\w-\psi)_+) -  \frac{C_1}{2} B_R(\w,(\w-\psi)_+^2)  = B_{R,\w}(\w,(\w-\psi)_+) ,
\]
where $B_{R,\w}$ is the bilinear form associated with the kernel
\[
K_{R,\w} (x,y) = K_R(x,y) \left[ 1 - \frac{C_1}{2}((\w-\psi)_+(x) + (\w-\psi)_+(y)) \right],
\]
which under \eqref{e:e0} satisfies similar bounds as the original kernel and is symmetric.  
Continuing with the energy inequality, we write $\w-\psi =(\w-\psi)_+ - (\w-\psi)_-$ and obtain
\begin{multline*}
B_{R,\w}(\w,(\w-\psi)_+) =B_{R,\w}((\w-\psi)_+,(\w-\psi)_+)  - B_{R,\w}((\w-\psi)_-,(\w-\psi)_+)\\
+ B_{R,\w}(\psi,(\w-\psi)_+).
\end{multline*}
The first is the main dissipative term for which we have a coercive bound
\[
B_{R,\w}((\w-\psi)_+,(\w-\psi)_+) \geq c_{\L,C_0} |(\w-\psi)_+|_{H^{1/2}}^2 - |(\w-\psi)_+|_2^2.
\]
For the second we have after cancellations
\[
- B_{R,\w}((\w-\psi)_-,(\w-\psi)_+) = 2 \int K_{R,\w}(x,y)(\w-\psi)_-(y)(\w-\psi)_+(z) \dy\dz: = P
\]
which is positive and can be dismissed for the application of the First DeGiorgi Lemma. Finally, as in \cite{CCV2011} we obtain
\[
|B_{R,\w}(\psi,(\w-\psi)_+)| \leq \frac12 B_R((\w-\psi)_+,(\w-\psi)_+) + |(\w-\psi)_+|_1 + |\{ \w-\psi >0\}|.
\]
We thus have proved the following energy bound under \eqref{e:e0} and for any rescaled solution with $RC_1>1$:
\[
 \ddt \int_\R (\w-\psi)_+^2 \dx +  |(\w-\psi)_+|_{H^{1/2}}^2 \lesssim  |(\w-\psi)_+|_2^2+ |(\w-\psi)_+|_1+|\{ \w-\psi >0\}|.
\]
We now recap the First DeGiorgi Lemma: there exists $\d>0$ and $\th \in (0,1)$ such that any solution $\w$ to \eqref{e:wR} satisfying
\[
\w(t,x) \leq 1 + ( |x|^{1/4} - 1)_+ \quad \text{on } \R \times [-2,0],
\]
and
\[
| \{ \w>0\} \cap (B_2 \times [-2,0])| \leq\d,
\]
must have a bound
\[
\w(t,x) \leq 1- \th.
\]
The proof proceeds as in \cite{CCV2011} with extra care given for \eqref{e:e0}. We consider Lipschitz function 
\[
\psi_{L_k}(x) = 1 - \th - \frac{\th}{2^k}  +   ( |x|^{1/2} - 1)_+.
\]
For $\th$ small enough it is clear that $(\w - \psi_{L_k})_+$ can be made as small as we wish for all $k\in \N$, in particular satisfying \eqref{e:e0}. With $\th$ fixed we can then apply the energy inequality for all terms $(\w - \psi_{L_k})_+$, and the argument of \cite{CCV2011} proceeds.

\medskip
\noindent  
\steps{4}{The second De Giorgi lemma}  In the Second De Giorgi Lemma the energy bound is used in a somewhat different way. Here the presence of the drift term requires extra attention as well as condition \eqref{e:e0}.  We recall the lemma first. For a $\l<1/3$ we define  $\psi_\l(x) = ((|x| - \frac{1}{\l^{4}})^{1/4}_+-1)_+$. Let also $F$ be non-increasing with $F =1$ on $B_1$ and $F = 0$ outside $B_2$. Define
\[
\phi_j = 1+ \psi_\l - \l^j F, \quad j=0,1,2.
\]
The lemma claims that there exist $\mu,\l,\g>0$ depending only on $\L$ such that if
\[
\w(t,x) <1+\psi_\l(x)  \text{ on } \R \times [-3,0],
\]
and 
\[
\begin{split}
|\{ \w <\phi_0\} \cap  B_1 \times (-3,-2) | & \geq \mu,\\
 |\{ \w >\phi_2\} \cap \R \times (-2,0) | & \geq \d,
\end{split}
\]
then necessarily
\[
|\{ \phi_0< \w <\phi_2\} \cap  \R\times (-3,0) |  \geq \g.
\]
So, if the function has a substantial weight under $\phi_0$ and later over $\phi_2$, then it must leave some appreciable weight in between. The proof goes by application of the energy inequality to
$(\w - \phi_1)_+$. However, $(\w - \phi_1)_+ \leq \l$ pointwise. Hence, to satisfy \eqref{e:e0} it is it sufficient to pick $\l <1/2C_0$, among further restrictions which come subsequently in the course of the proof. Thus, we have
\[
\begin{split}
\ddt \int_\R (\w-\phi_1)_+^2 \dx &+ B_{R,\w}((\w-\phi_1)_+,(\w-\phi_1)_+)  + P = - B_{R,\w}(\phi_1,(\w-\phi_1)_+) \\
&+ \int \left( \frac12 u_R (\phi_1)_x + f_{R,C_1} \right) (\w-\phi_1)_+ \dx .
\end{split}
\]
All the terms are exactly the same as in \cite{CCV2011} except the last one. To bound the last term we note that $ (\w-\phi_1)_+ $ is supported on $B_2$, where $\phi_1 = 1 + \l F$, hence $|(\phi_1)_x|_{L^\infty(B_2)} \leq C \l$. Furthermore, as noted above, $(\w-\phi_1)_+ \leq \l$.  Hence,
\[
 \left| \frac12 \int u_R (\phi_1)_x (\w-\phi_1)_+ \dx  \right| \leq C \l^2.
\]
As to the source term, we obtain the same bound provided $\e_0 <\l$.  The resulting bound repeats another estimate on the term $B_{R,\w}(\phi_1,(\w-\phi_1)_+)$, and hence, blends with the rest of Section 4 in \cite{CCV2011}.

The rest of the proof makes no further direct use of the energy inequality and thus proceeds ad verbatim.   The penultimate constant $\l$ ends up being dependent only on $\L$ and $C_0$ which are scaling invariant.

\medskip
\noindent  
\steps{5}{Diminishing oscillation and $C^\g$ regularity}   The first and second De Giorgi lemmas are now being used to prove that any solution with controlled tails on $[-3,0]\times \R$,
\[
-1 - \psi_{\e,\l} \leq w \leq 1+ \psi_{\e,\l},
\]
where 
\[
\psi_{\e,\l}(x) = \begin{cases}
0&, \quad \text{ if } |x| < \l^{-4} \\
[(|x| - \l^{-4})^\e - 1]_+&, \quad \text{ if } |x| \geq \l^{-4}
\end{cases}
\]
satisfies
\[
\sup_{[-1,0]\times B_1}  w - \inf_{[-1,0]\times B_1}  w  < 2 - \l^*,
\]
for some $\l^*>0$.  The proof goes by application of shift-amplitude rescalings of the form
\[
w_{k+1} = \frac{1}{\l^2}( w_k - (1-\l^2)) = \frac{1}{\l^{2k}} w   +  C_k.
\]
For our sourced equation this is the worst kind of rescaling since it doesn't come with a compensated space-time stretching.  However, in the argument the number of iterations is limited to $k_0 = |[-3,0]\times B_3|/\g$, and hence depends only on $\L$. We can pre-scale the equation in the beginning using $R'>0$ so large that $\e_0 = |f_{R'}|_\infty   <  \l^{2k_0} C_0 \leq \l^{2k_0}$. Hence, on each step of the iteration we have $|f_k|<\l$, fulfilling the requirement of Step 4 automatically.

The final iteration  consists on zooming and shifting process:
\[
\begin{split}
w_1 & = w / |w|_\infty ,\\
w_{k+1} & = \frac{1}{1 - \l^*/4} ((w_k)_R - \bar{w}_k),
\end{split}
\]
where $\bar{w_k}$ is the average over $[-1,0]\times B_1$.  On the first step we still have the bound $|f_1|<\l^{2k_0}$. Subsequently,  among other restrictions put on $R$ in \cite{CCV2011} we set in addition $R (1 - \l^*/4) >1$, which preserves the bound $|f|<\e_0$ for all steps.    This finishes the proof.

\section{Further extensions and discussion}

The class of topological models can be extended within our framework to include generalized topological diffusion of type 
\begin{equation}\label{eq:tau}
\phi(\bx,\by)=\frac{h(|\bx-\by|)}{|\bx-\by|^{n+\a-\tau}}\times \frac{1}{\dd_\rho^\tau(\bx,\by)}, \qquad \tau > 0.
\end{equation}
In, fact this class arises naturally in a hierarchy fashion in commutator estimates proved below in Appendix B. Our main flocking result of \thm{t:alignrate-a} extends to all $\tau \geq n$. In fact the most general statement which includes various stronger assumptions on density, and hence, better alignment rates, can be summarized in the following formulation.

\begin{theorem}\label{t:alignrate-b}  Let $(\rho,\bu)$ be a  global smooth solution of the topological model with kernel  \eqref{eq:tau}. Assume that  the density $\rho(t,\cdot)$ satisfies, for all $t>0$,
	\begin{equation}\label{e:ralg}
	\rho(t,\bx) \geq  \frac{c}{(1+t)^\b}, \quad 0\leq \b \leq \b_0:= \min\left\{1, \frac{n}{2n - \tau}\right\},
	\end{equation}
	and if $\tau > n+\a$, additionally
	\begin{equation}\label{e:rhoup}
	|\rho(t,\cdot)|_{\frac{\tau -n}{\a}} < C.
	\end{equation}
	Then the solution aligns with at least algebraic rate given by
	\begin{equation}\label{e:alignrate-b}
	|\bu(t) - \bu_\infty |_\infty  = \frac{o(1)}{t^\g} \ \text{ where } \ \g = \frac12\left( 1 - \frac{\b}{\b_0} \right).
	\end{equation}
\end{theorem}

One notable application of this more general result is for the 1D case when $e\equiv 0$. Indeed, in this case we have a uniform bound on the density from above and below, see \eqref{e:rhomaxpr} and hence the alignment rate improves to $\g =\frac12$.

More can be said about the density itself.  If $e=0$ the continuity equation acquires the structure of the $u$-equation. Along with the maximum principle come a possibility of applying \thm{t:alignrate-b} directly to the continuity equation. The energy law takes form
\[
\ddt |\rho|_2^2 = \int |\rho|^2 \cL_\phi \rho  \dx,
\]
which after symmetrizing becomes
\[
\int |\rho|^2 \cL_\phi \rho  \dx =- \frac12 \int \phi(x,y) (\rho(x)+\rho(y)) (\rho(x) - \rho(y))^2 \dx \dy.
\]
Since the pre-factor $(\rho(x)+\rho(y))$ is uniformly bounded from above and below this supplies the energy inequality analogous to \eqref{e:eneq}. We have all ingredients for a direct application of \thm{t:alignrate} (with $\b=0$) to the continuity equation and we conclude
\[
	|\rho(t) - \frac{1}{2\pi}M_0 |_\infty  = \frac{o(1)}{\sqrt{t}}.
\]

\begin{remark}({\bf About} $\tau=n$). We make another remark concerning the apparent threshold value of $\tau=n$. Clearly from \eqref{e:ralg}, if $\tau \geq n$, then $\rho \geq \frac{1}{1+t}$ is the weakest assumption under which the theorem holds, while for $\tau<n$ a more stringent bound on $\rho$ is required.  This can be explained by the fact the the density on the bottom of $\phi$ needs to compensate the density  on the top inside the diffusion term. Even more vividly the condition manifests itself after taking limit as $\a \to 2$. Such limits are standard in the  elliptic theory and we will not provide many details here. One can verify the following:
	\begin{equation}\label{e:classlimit}
	\lim_{\a \to 2} (2-\a) \cL_\phi f(x) = \n \cdot \left( \rho^{-\frac{\tau}{n}} \n f \right)  : = D(f) .
	\end{equation}
	The commutator which would appear in the corresponding limit model reads
	\begin{equation}\label{e:commD}
	D(\rho \bu) - \bu D(\rho) =  \frac{1}{\rho^{\g - 1}} \Delta \bu  + \frac{2-\g }{\g}  \n \bu \n \rho, \qquad \g = \frac{\tau}{n}.
	\end{equation}
	We can see  that $\tau = n$ is the threshold that determines whether   the density appears on the top or the bottom  in front of the leading order term. For $\tau \geq n$ it amplifies dissipation in thinner regions as intended in the topological model.  
\end{remark}

Concerning regularity of solutions in 1D one can obtain an extension into the range $\a<1$. In fact the continuation criterion of \prop{p:global}  extends directly as is, in fact in several technical places even easier due to lower singularity order of the diffusion.  The H\"older regularization result can be obtained by an adaptation of Silvestre result \cite{S2012} for forced drift-diffusion equations. The result assume pure fractional Laplacian as a diffusion, but as noted by the author, applies to more general kernels, even in $z$:  $K(x,z,t) = K(x,-z,t)$. Another necessary condition to apply \cite{S2012} is regularity of the drift $u \in C^{1-\a}$.   For this we use the representation \eqref{e:Sch1}:  $u =  \p_x^{-1} e - \cF \rho$. Since $\p_x^{-1} e\in W^{1,\infty}$, it remains to check that $\cF \rho \in C^{1-\a}$. The verification again goes via an optimization over cut-off scale argument.  Then, omitting constants, 
\[
\begin{split}
\cF \rho(x+\xi) - \cF \rho(x) & = \int_{|z|\geq |\xi|} [ \ln \dd_\rho(x+\xi+z,x+\xi) - \ln \dd_\rho(x+z,x)] \frac{\sgn(z)  h(z)}{|z|^{\a}} \dz\\
& + \int_{|z|\leq |\xi|}  [ \ln \dd_\rho(x+\xi+z,x+\xi) - \ln \dd_\rho(x+z,x)] \frac{\sgn(z)  h(z)}{|z|^{\a}} \dz.
\end{split}
\]
In the first, we use Taylor formula \eqref{e:taylorlog} which yields a bound by $|\xi|/ |z|^{1+\a}$, with a uniform constant depending only on \eqref{e:cC}. This results in  $|\xi|^{1-\a}$, as needed. 
In the latter integral we simply observe
\[
\ln \dd_\rho(x+\xi+z,x+\xi) - \ln \dd_\rho(x+z,x) = \ln \frac{ \dd_\rho(x+\xi+z,x+\xi)}{\dd_\rho(x+z,x) }  \sim 1.
\]
So, the order of singularity is $|z|^{-\a}$, which implies bound by $|\xi|^{1-\a}$, as needed. 

A restriction comes in the range $0<\a<1$, or $\a<\tau$ for more general models, in establishing upper bound on the density. While the lower bound in \eqref{e:rhouplow} always holds, the extension to upper bound reads as follows.

\begin{lemma}\label{l:above2} Let $(\rho,u)$ be a smooth solution of the $(\tau,\a)$-model, subject to initial density $\rho_0$ away from vacuum, $0<c < \rho_0 < C <\infty$. Assume that either (i) $\tau \leq \a$, or else if $\tau > \a$, that (ii) the initial condition satisfies 
	\[
	\cM_0^{\tau} |q_0|_\infty < \frac{R_0^{\tau-\a}}{\tau-\a}, \quad q_0=\frac{e_0}{\rho_0}.
	\] 
	Then the density is uniformly bounded in time:
	\begin{equation}\label{e:above2}
	\rho(t,x)<C(\cM_0,|q_0|_\infty,\phi), \qquad  x\in\T, \ t \geq 0.
	\end{equation}
\end{lemma}
So, for $\tau > \a$ we need an extra smallness assumption to acheive the same result. This condition is scaling invariant, see Step 2 in the proof of \prop{p:degiorgi}.  We record the generalization in the following theorem.

\begin{theorem}\label{t:mainI2}
	Consider the one-dimensional system \eqref{e:main} on $\T$ with short-range topological kernel \eqref{eq:tau} and  singularity of order $0<\a<1$. Any non-vacuous initial data $(\rho_0,u_0)\in H^{s+\a}\times H^{s+1}$, $s\geq 3$, satisfying the conditions of \lem{l:above2} admits 
	a unique global in time solution,  $(\rho,u)$, in the class  
	\[
	\begin{split}
	\rho & \in C_w(\R^+; H^{s+\a})\cap L^2_\loc(R^+; H^{s+1+\frac{\a}{2}}), \\
	u & \in C_w(\R^+; H^{s+1})\cap L^2_\loc(R^+; H^{s+1+\frac{\a}{2}}).
	\end{split}
	\]
\end{theorem}

%%%%%%%%%%%%%%%%%%%%%
\section{Appendix A. Pointwise evaluation of topological alignment}\label{appA}
%%%%%%%%%%%%%%%%%%%%}
%%%%%%%%%%%%%%%%%%%%%%%%%%
Here we collect necessary formalities related to pointwise evaluations of the operator $\cL_\phi$ and the commutator $\cC_\phi$. The statements come with corresponding estimates we used throughout the text. In fact, we consider the more general class of topological kernels that we already mentioned in the previous section:
\begin{equation}\label{e:taukernel}
	\phi(\bx + \bz,\bx) = \frac{h(|\bz|)}{  |\bz|^{n+\a-\tau}}\times \frac{1}{\dd^{\tau}_\rho(\bx+\bz,\bx)}, \quad \tau >0.
\end{equation}

\begin{lemma} \label{l:topoL}
	For any $0<\a<2$ and $f\in C^2$ one has the natural pointwise representation formula
	\begin{equation}\label{e:L=Ltopo}
	\aL_\phi f(\bx)=p.v.\int_{\T^n}  (f(\bx+\bz)-f(\bx))\phi(\bx+\bz,\bx)\dbz.
	\end{equation}
	Moreover, for any $r>0$,
	\begin{equation}\label{e:repL}
	\aL_\phi f(\bx)=\int_{\T^n}  (f(\bx+\bz)-f(\bx) -\bz \cdot\n f(\bx)  \one_{|\bz|<r}(\bz) )\phi(\bx+\bz,\bx)\dbz + b_r(\bx)\cdot \n f(\bx),
	\end{equation}
	where 
		\[
	b_r (\bx) = p.v. \int_{|\bz|<r}  \bz  \phi(\bx+\bz,\bx) \dbz,
	\]
	satisfies
	\begin{equation}\label{e:bN}
	|b_r|_\infty  \leq C |\n \rho|_\infty r^{2-\a}.
	\end{equation}
\end{lemma}
\begin{proof}
	At the core of the proof is a bound on the operator given by 
	\[
	B_r\zeta (\bx) = p.v. \int_{|\bz|<r} \zeta(\bx+\bz)\, \bz  \phi(\bx+\bz,\bx) \dbz.
	\]
	Clearly,  $B_r 1 = b_r$. We address it more generally as was used in preceding sections. By symmetrization,
	\[
	\begin{split}
	B_r\zeta (\bx) &= \frac12  \int_{|\bz|<r} \frac{\dd^\tau_\rho(\bx-\bz,\bx) - \dd^\tau_\rho(\bx+\bz,\bx)}{\dd^\tau_\rho(\bx+\bz,\bx)\dd^\tau_\rho(\bx-\bz,\bx) |\bz|^{n+\a-\tau}} \zeta(\bx+\bz) \bz h(\bz) \dbz\\
	& + \frac12  \int_{|\bz|<r} \frac{\zeta(\bx+\bz) - \zeta(\bx-\bz)}{\dd_\rho^\tau(\bx-\bz,\bx) |\bz|^{n+\a-\tau}}\bz h(\bz) \dbz =: I(\bx)+J(\bx).
	\end{split}
	\]
	In what follows the constant $C$ will change line to line and may depend on the underlying bounds on the density at hand, \eqref{e:rbounds}. As for $J$, we directly obtain
	\[
	|J(\bx) | \leq C |\n \zeta |_\infty r^{2-\a}.
	\]
	For $I(\bx)$ we first observe 
\[
	\begin{split}
		\dd^\tau_\rho(\bx+\bz,\bx) - \dd^\tau_\rho(\bx-\bz,\bx) & = \frac{\tau}{n} [ \dd^{n}_\rho(\bx+\bz,\bx) - \dd^{n}_\rho(\bx-\bz,\bx)]\times \\
		&\times  \int_0^1 \left[ \th \dd^{n}_\rho(\bx+\bz,\bx) +(1-\th)\dd^{n}_\rho(\bx-\bz,\bx) \right]^{\frac{\tau}{n} - 1} \dth.
	\end{split}
\]
	Note that
	\[
	| \dd^{n}_\rho(\bx+\bz,\bx) - \dd^{n}_\rho(\bx-\bz,\bx) | = \left| \int_{\O(\bz,0)} (\rho(\bx+\bw) - \rho(\bx-\bw))\dbw \right| \leq |\n \rho|_\infty |\bz|^{n+1},
	\]
	and clearly,
	\[
	\int_0^1 \left[ \th \dd_\rho(\bx+\bz,\bx) +(1-\th)\dd_\rho(\bx-\bz,\bx) \right]^{\frac{\tau}{n} - 1} \dth \leq C|\bz|^{\tau - n}.
	\]
	Consequently,
	\[
	|I(\bx)| \leq  C |\n \rho|_\infty|\zeta|_\infty \int_{|\bz|<r} \frac{1}{|\bz|^{n+\a -2}}\dbz \sim |\n \rho|_\infty|\zeta|_\infty r^{2-\a}.
	\]
	
	In conclusion we obtain the bound
	\begin{equation}\label{e:BN}
	| B_r\zeta |_\infty  \leq C \left(|\n \rho|_\infty |\zeta|_\infty+ |\n \zeta|_\infty   \right) r^{2-\a}.
	\end{equation}
	Note that the bounds above provide a common integrable dominant for the integrands parametrized by $\bx$. So, in addition $B_r\zeta \in C(\T^n)$.
	
	The bound \eqref{e:bN} now follows directly from \eqref{e:BN}, and we also have $b_r\in C(\T^n)$.
	With the knowledge that the drift is finite, clearly, the right hand sides of \eqref{e:L=Ltopo} and \eqref{e:repL} coincide. Denote them $L_\phi f(\bx)$. We now have a task to pass to the limit
	\[
	\lan    \cL_\phi f , g_\e \ran \ra L_\phi f(\bx_0),
	\]
	for \emph{every} $\bx_0\in \T^n$. Splitting the integral we obtain
	\[
	\begin{split}
	\lan    \cL_\phi f , g_\e \ran & = \frac12 \int_{\T^n \times \T^n} \phi(\bx,\by) (f(\by)-f(\bx) - \n f(\bx)(\by-\bx) \one_{|\bx-\by|<r})(g_\e(\bx) - g_\e(\by))\dby \dbx \\
	&+  \frac12 \int_{\T^n \times \T^n} \phi(\bx,\by) \n f(\bx)(\by-\bx) \one_{|\bx-\by|<r})(g_\e(\bx) - g_\e(\by))\dby \dbx = I + J.
	\end{split}
	\]
	Note that $J= \frac12 \lan b_r \cdot \n f ,  g_\e \ran + \frac12 \lan B_r \n f , g_\e \ran$. By continuity of $B_r$ proved above,
	\begin{equation}\label{e:Jcomp}
	J \to  \frac12  b_r(\bx_0)\cdot \n f(\bx_0) + \frac12 (B_r \n f)(\bx_0).
	\end{equation}
	
	As for  $I$ we can unwind the symmetrization since each part of the integral is not singular any more:
	\[
	\begin{split}
	I & = \frac12 \int_{\T^n \times \T^n} \phi(\bx,\by) (f(\by)-f(\bx) - \n f(\bx)(\by-\bx) \one_{|\bx-\by|<r})g_\e(\bx) \dby \dbx \\
	& - \frac12 \int_{\T^n \times \T^n} \phi(\bx,\by) (f(\by)-f(\bx) - \n f(\bx)(\by-\bx) \one_{|\bx-\by|<r})g_\e(\by) \dby \dbx.
	\end{split}
	\]
	Passing to the limit in each integral we obtain
	\[
	\begin{split}
	I & \ra \frac12 \int_{\T^n} (f(\by) - f(\bx_0) - \n f(\bx_0)(\by-\bx_0)) \phi(\bx_0,\by) \dby  \\
	& -  \frac12 \int_{\T^n} (f(\bx_0) - f(\bx) - \n f(\bx)(\bx_0-\bx))\phi(\bx,\bx_0)\dbx\\
	& =  \int_{\T^n} \phi(\bx_0,\by) (f(\by)-f(\bx_0) - \frac12 (\n f(\bx_0) + \n f(\by))(\by-\bx_0) \one_{|\bx_0-\by|<r})\dby\\
	&= \int_{\T^n} \phi(\bx_0,\by) (f(\by)-f(\bx_0) - \n f(\bx_0) (\by-\bx_0) \one_{|\bx_0-\by|<r})\dby \\
	&+  \frac12 \int_{\T^n} \phi(\bx_0,\by) (\n f(\bx_0) - \n f(\by))(\by-\bx_0) \one_{|\bx_0-\by|<r}\dby \\
	&= L_\phi f(\bx_0) - \frac12  b_r(\bx_0)\cdot \n f(\bx_0) - \frac12 (B_r \n f)(\bx_0).
	\end{split}
	\]
	Thus, combining with \eqref{e:Jcomp} we obtain $	I+ J \to  L_\phi f(\bx_0)$ which completes the proof. 
\end{proof}

As a corollary we obtain analogous representation formula for the commutator.

\begin{lemma} \label{l:topoC}
	For any $0<\a<2$ one has the following pointwise representation
	\begin{equation}\label{e:C=Ctopo}
	\cC_\phi(f,\zeta)(\bx)=p.v.\int_{\T^n} \phi(\bx+\bz,\bx) \zeta(\bx+\bz) (f(\bx+\bz)-f(\bx))\dbz.
	\end{equation}
	Moreover, the following representation holds for any $r>0$:
	\begin{equation}\label{e:repC}
	\begin{split}
	\cC_\phi(f,\zeta)(\bx)&=\int_{\T^n} \phi(\bx+\bz,\bx)\zeta(\bx+\bz) (f(\bx+\bz)-f(\bx) - \bz\cdot \n f(\bx) \one_{|\bz|<r})\dbz \\
	&+ (\zeta(\bx) b_r(\bx) + a_r(\bx)) \cdot \n f(\bx),
	\end{split}
	\end{equation}
	where $b_r$ is defined as before, and 
	\begin{equation}\label{e:ar}
	|a_r|_\infty \leq C |\n \zeta|_\infty r^{2-\a}.
	\end{equation}
\end{lemma}
The proof goes by a direct application of \lem{l:topoL}. For the residual drift we obtain
\[
a_r(\bx) = \int_{|\bz|<r} \phi(\bx+\bz,\bx)(\zeta(\bx+\bz) - \zeta(\bx)) \bz \dbz.
\]
The bound \eqref{e:ar} follows at once.

\section{Appendix B. Commutator estimates}\label{appB}

We will forcus on 1D case with $\a \geq 1$ and establish necessary commutator estimates used in the proof of \thm{t:mainI}.  The estimates will be obtained in pointwise evaluation style which makes them suitable for applications in both $L^\infty$-based settings and $L^2$ setting. For this reason we pay special attention to dependence on the top order terms.  First, we obtain a basic estimate on pointwise evaluation of the topological diffusion operator, which follows  from representation formula \eqref{e:repL}.
\begin{lemma}\label{l:Linfty}  For every smooth function $f$ and $0\leq \g <1$ one has
	\begin{equation} \label{e:Linfty}
	| \cL_\phi f (x) | \lesssim  r^{1 - \frac{\a}{2}}\sqrt{\dD_\a[ f'] (x)} + r^{\g - \a} \|f\|_{C^\g} + r^{2-\a} |f'(x)| |\rho'|_\infty,
	\end{equation}
for all $r<R_0$ and $x\in \T$. 
\end{lemma}
\begin{proof}
 We use decomposition \eqref{e:repL} with further breakdown of the integral:
\[
\begin{split}
\cL_\phi f(x) & = \int_{|z|<r} (f(x+z) - f(x) - z f'(x)) \phi \dz +  f'(x)b_r(x) \\
&+ \int_{|z|>r} (f(x+z) - f(x) )\ \phi \dz = I+f'(x)b_r(x)+J.
\end{split}
\]
Using that 
\begin{equation}\label{e:diffD}
|f(x+z) - f(x) - z f'(x)| = \left| \int_0^z (f'(x+w) - f'(x)) \dw \right| \lesssim \sqrt{\dD_\a f'(x)} |z|^{1+ \frac{\a}{2}},
\end{equation}
we obtain
\[
|I|\lesssim r^{1-\frac{\a}{2}} \sqrt{\dD_\a f'(x)}.
\]
Next, due to \eqref{e:bN}, 
\[
|b_r(x)|\lesssim |\rho'|_\infty r^{2-\a}.
\]
And as to $J$, we use the H\"older continuity,
\[
|J| \lesssim r^{\g - \a} \|f\|_{C^\g}.
\]
Putting the estimates together yields \eqref{e:Linfty}.
\end{proof}

\begin{lemma}\label{l:1stcomm}  For every smooth function $f$ and $0\leq \g <1$ one has 
	\begin{multline} \label{e:1stcomm}
	| \cL_{\phi'} f (x) | \lesssim  r^{1 - \frac{\a}{2}} \left( |\rho'|_\infty \sqrt{\dD_\a[ f'] (x)}+  |f'|_\infty \sqrt{\dD_\a [\rho'](x)}\right)+ r^{\g - \a} \|f\|_{C^\g}  |\rho'|_\infty \\
	+ r^{2-\a} |f'|_\infty |\rho'|^2_\infty,
	\end{multline}
for all $r<R_0$ and $x\in \T$. 
\end{lemma}
\begin{proof}
Using the explicit formula for the kernel
\[
\phi' = \frac{h(z)}{|z|^{\a} \dd_\rho^{2}(x,x+z)} \int_{[0,z]} \rho'(x+ \xi) \dxi,
\]
we obtain
\[
\begin{split}
\cL_{\phi'} f (x) & = \int_\T\frac{h(z) \d_z f(x) }{|z|^{\a} \dd_\rho^{2}(x,x+z)} \int_{[0,z]} [ \rho'(x+ \xi) - \rho'(x) ] \dxi     \dz\\
&+  \rho'(x) \int_\T\frac{h(z)}{|z|^{\a - 1} \dd_\rho^{2}(x,x+z)}   \d_z f(x)   \dz = J_1 + J_2.
\end{split}
\]
Note that $J_2$ is precisely one of the topological-type operator with $\tau = 2$. So, estimate \eqref{e:Linfty} applies:
\[
|J_2| \lesssim  r^{1 - \frac{\a}{2}}   |\rho'|_\infty  \sqrt{\dD_\a[ f'] (x)} + r^{\g - \a} \|f\|_{C^\g}   |\rho'|_\infty + r^{2-\a} |f'|_\infty   |\rho'|^2_\infty.
\]
As to $J_1$, we estimate as usual
\[
\left| \int_{[0,z]} [ \rho'(x+ \xi) - \rho'(x) ] \dxi  \right| \lesssim \sqrt{\dD_\a [\rho'](x)} |z|^{1+ \frac{\a}{2}}.
\]
So, using this together with the full derivative of $f$  in the short range $\{|z|<r\}$, and H\"older continuity of $f$ in the long range $\{|z| >r\}$ we obtain
\[
|J_1| \lesssim r^{1 - \frac{\a}{2}} |f'|_\infty \sqrt{\dD_\a [\rho'](x)} +  r^{\g - \a} \|f\|_{C^\g} |\rho'|_\infty .
\]

\end{proof}

The statement of \lem{l:1stcomm} can be viewed as the commutator estimate of first order since
\[
\cL_{\phi'} f = (\cL_\phi f)' - \cL_{\phi} f'.
\]
We will need to establish similar estimates for higher order commutators, although without the use of H\"older continuity of $f$.
\begin{lemma}\label{l:kcomm} Let $f,\rho$ be smooth functions and $1\leq \a <2$. Then for any $x\in \T$ the following inequalities hold: for $k \geq 3$
\begin{equation}\label{e:k3}
			\begin{split}
				| (\cL_\phi f)^{(k)} (x) -  \cL_\phi (f^{(k)}) (x) |  &\lesssim \sqrt{\dD_\a[ f^{(k)}] (x)}  + \sqrt{\dD_\a[ \rho^{(k)}] (x)} \\
			& +  \sqrt{\dD_\a[ f^{(k-1)}] (x)} + \sqrt{\dD_\a[ \rho^{(k-1)}] (x)} \\ 
			& + |\rho^{(k)}(x) |  + |f^{(k)}(x)|  +  1,
			\end{split}
\end{equation}
and for any $\e>0$, and $k=2$
\begin{equation}\label{e:k2}
	\begin{split}
| (\cL_\phi f)'' (x) -  \cL_\phi (f'') (x) |  &\lesssim \sqrt{\dD_\a[ f''] (x)}  + \sqrt{\dD_\a[ \rho''] (x)} \\
& + \e |f''|_\infty |\rho''|_\infty+  |f''|_\infty  +  c_\e |\rho''|_\infty +c_\e.
\end{split}
\end{equation}
with $\lesssim$ meaning  up to a contant factor 
\[
C = C(\rhomin,\rhomax,|\rho'|_\infty,|f'|_\infty,\ldots , |f^{(k-1)}|_\infty, |\rho^{(k-1)}|_\infty).
\]
\end{lemma}
\begin{proof}
	According to \eqref{e:LRn} we have to obtain estimates on all terms 
	\[
	 \cL_{\phi^{(l)}} [f^{(k - l)}] (x), \text{  for } 1 \leq l \leq k.
	\]
The kernel can be expanded using the Faa di Bruno formula
\[
\phi^{(l)} =\sum_{\bj} C_\bj \frac{h(z)}{|z|^{\a} \dd_\rho^{1 + |\bj|}(x,x+z)} \prod_{p = 1}^l \left[ \int_{[0,z]} \rho^{(p)}(x+ \xi) \dxi \right]^{j_p},
\]
where $\bj = (j_1,...,j_l)$ is a multi-index with weight $|\bj|= j_1 + ... + j_l$, and 
\[
1j_1 + 2 j_2 + \ldots + l j_l = l.
\]
Let us introduce into consideration operators corresponding to the summands in the above expansion
\[
\cL_\bj f^{(k-l)}(x) =   \int_\T\frac{h(z) \d_z f^{(k-l)}(x) }{|z|^{\a} \dd_\rho^{1+|\bj|}(x,x+z)} \prod_{p = 1}^l \left[ \int_{[0,z]} \rho^{(p)}(x+ \xi) \dxi \right]^{j_p}  \dz .
\]
Let us consider separately one end-point case when the index reaches its corner value $\bj = (0,...,0,1)$. For this particular index the density receives its maximal derivative:
\[
\begin{split}
\cL_\bj f^{(k-l)}(x) & = \int_\T\frac{h(z)   \d_z f^{(k-l)}(x) }{|z|^{\a} \dd_\rho^{2}(x,x+z)} \int_{[0,z]} \rho^{(l)}(x+ \xi)\dxi     \dz \\
& = \int_\T\frac{h(z)   \d_z f^{(k-l)}(x) }{|z|^{\a} \dd_\rho^{2}(x,x+z)} \int_{[0,z]} (\rho^{(l)}(x+ \xi) - \rho^{(l)}(x))\dxi     \dz \\
&+ \rho^{(l)}(x) \int_\T\frac{h(z)   \d_z f^{(k-l)}(x) }{|z|^{\a-1} \dd_\rho^{2}(x,x+z)} \dz = J_1 + J_2.
\end{split}
\]
The operator involved in $J_2$ is exactly of topological type with $\tau = 2$. So, we apply \eqref{e:Linfty} directly with $r \sim \e$ and $\g = 0$:
\[
|J_2| \lesssim |\rho^{(l)}(x) | \left( \e \sqrt{\dD_\a[ f^{(k-l+1)}] (x)} + c_\e + |f^{(k-l+1)}(x)| \right).
\]
Here and in the following we dismiss all the quantities depending on the  lower order terms $\rhomin,\rhomax,|\rho'|_\infty,|f'|_\infty,\ldots , |f^{(k-1)}|_\infty, |\rho^{(k-1)}|_\infty$. Using 
\[
 \dD_\a[ f^{(k-l+1)}] (x) \lesssim  |f^{(k-l+2)}|_\infty
\]
we see that all the terms with $l=3,...,k-1$ are of lower order (this only applies if $k \geq 4$).  For $l=k$ we have
\[
|J_2| \lesssim \e|\rho^{(k)}(x) | |f''|_\infty + c_\e.
\]
For $l=1$, we have
\[
|J_2| \lesssim  \sqrt{\dD_\a[ f^{(k)}] (x)}  +  |f^{(k)}(x)|.
\]
For $l=2$,
\[
|J_2| \lesssim \e |\rho''|_\infty  \sqrt{\dD_\a[ f^{(k-1)}] (x)} +  c_\e |\rho''|_\infty .
\]
Summing up over $l$ we have
\[
\sum_{l=1}^k |J_2| \lesssim \sqrt{\dD_\a[ f^{(k)}] (x)}  +  |f^{(k)}(x)| + \e|\rho^{(k)}(x) | |f''|_\infty +\e |\rho''|_\infty  \sqrt{\dD_\a[ f^{(k-1)}] (x)} +  c_\e |\rho''|_\infty +c_\e.
\]
As to $J_1$ terms, for all $2\leq l\leq k-2$ we simply estimate 
\[
|J_1|\lesssim |f^{(k-l+1)}|_\infty|\rho^{(l+1)}|_\infty \lesssim C.
\]
For $l=1$,
\[
|J_1|\lesssim  |\rho''|_\infty \int_{|z| \leq \e} \frac{|\d_z f^{(k-1)}(x)| }{|z|^{\a}} \dz + c_\e 
|f^{(k-1)}|_\infty  |\rho'|_\infty  \lesssim \e^{1-\a/2}  |\rho''|_\infty \sqrt{\dD_\a[ f^{(k-1)}] (x)} + c_\e.
\]
Resetting $\e^{1-\a/2} \to \e$ this term has been accounted for. For $l=k-1$, 
\[
|J_1|\lesssim \e |f''|_\infty  \sqrt{\dD_\a[ \rho^{(k-1)}] (x)} +c_\e.
\]
Finally, for $l=k$, we have
\[
|J_1| \lesssim \sqrt{\dD_\a[ \rho^{(k)}] (x)} .
\]
To summarize, the corner-case terms add up to 
\begin{equation}\label{e:capture}
	\begin{split}
	\sum_{l=1}^k |\cL_\bj f^{(k-l)}(x) | &\lesssim \sqrt{\dD_\a[ f^{(k)}] (x)}  + \sqrt{\dD_\a[ \rho^{(k)}] (x)} \\
	& + \e |f''|_\infty  \sqrt{\dD_\a[ \rho^{(k-1)}] (x)} + \e |\rho''|_\infty  \sqrt{\dD_\a[ f^{(k-1)}] (x)} \\ 
	& + \e|\rho^{(k)}(x) | |f''|_\infty + |f^{(k)}(x)|  +  c_\e |\rho''|_\infty +c_\e.
	\end{split}
\end{equation}

Let us now consider off-corner cases, $\bj = (j_1,...,j_{l-1},0)$, $2 \leq l \leq k$ (obviously for $l=1$ there is only one term with $\bj = (1)$ which is a corner case). Since $|\d_z f^{(k-l)}| \leq |z| |f^{(k-1)}|_\infty \lesssim C|z|$, the order of singularity of the kernel becomes $\a + |\bj|$, while the order of the product in the numerator is ${|\bj|}$. So, for $\a \geq 1$ this operator is still hypersingular, which means extra care needed  find additional cancellations. Let us denote
\[
a_p = \int_{[0,z]} \rho^{(p)}(x+ \xi) \dxi, \quad b_p = |z| \rho^{(p)}(x),
\]
and write the product as follows
\[
\begin{split}
	\prod_{p = 1}^{l-1} a_p^{j_p}  & =  a_1^{j_1}\cdots a_{l-2}^{j_{l-2}}  ( a_{l-1}^{j_{l-1}} - b_{l-1}^{j_{l-1}} ) +  a_1^{j_1}\cdots a_{l-3}^{j_{l-3}}  (a_{l-2}^{j_{l-2}} -  b_{l-2}^{j_{l-2}} ) b_{l-1}^{j_{l-1}}   + \ldots \\
	&+ (a_1^{j_1}-b_1^{j_1}) b_2^{j_2}  \cdots  b_{l-1}^{j_{l-1}}  + 	\prod_{p = 1}^{l-1} b_p^{j_p} .
\end{split}
\]
Now, for $p\leq k-2$ we simply use
\[
|a_{p}^{j_{p}} - b_{p}^{j_{p}}| \leq |z|^{1+j_p} |\rho^{(p+1)}|_\infty |\rho^{(p)}|^{j_p-1}_\infty  \lesssim  |z|^{1+j_p} .
\]
So, the product in this case is bounded by $\lesssim |z|^{1+|\bj|}$, and the singularity order reduces to $\a-1 <1$. This, these terms are bounded by $\lesssim C$. 

For $p=k-1$, if $j_{k-1} >0$ we use
\[
|a_{k-1}^{j_{k-1}} - b_{k-1}^{j_{k-1}}| \lesssim |z|^{\frac{\a}{2}+j_{k-1}} \sqrt{ \dD_\a \rho^{(k)}(x)}.
\]
Thus, the order of the product is $\frac{\a}{2}+|\bj|$, and the order of the operator becomes $\a/2<1$. So, this term is bounded by $\lesssim \sqrt{ \dD_\a \rho^{(k)}(x)}$, which has been accounted for earlier. 

 It remains to estimate the integral for the pure $b$-product:
\[
\prod_{p = 1}^{l-1} b_p^{j_p} = |z|^{|\bj|} \prod_{p = 1}^{l-1} (\rho^{(p)}(x))^{j_p}.
\]
The product of densities is obviously subcritical and comes out of the integral. What remains is another topological operator 
\[
\int_\T\frac{h(z) \d_z f^{(k-l)}(x) }{|z|^{\a - |\bj|} \dd_\rho^{1+|\bj|}(x,x+z)}  \dz.
\]
This involves the generalized kernel \eqref{e:taukernel} with $\tau = 1+|\bj|$. Applying estimate \eqref{l:Linfty} with $\g=0$ and fixed absolute $r \sim 1$ we obtain the bound
\[
\lesssim  \sqrt{\dD_\a[ f^{(k-l+1)}] (x)} +  |f^{(k-l)}|_{\infty} + |f^{(k-l+1)}(x)|.
\]
Recalling that we are in the range $2\leq l\leq k$, we have 
\[
|f^{(k-l)}|_{\infty} + |f^{(k-l+1)}(x)|\lesssim 1,
\]
while for $l>2$ the dissipative term is also subcritical, and for $l=2$, term $\sqrt{\dD_\a[ f^{(k-1)}] (x)} $ has been accounted for.

Thus, estimate \eqref{e:k3} captures all the terms we encountered. It remains to notice that for $k \geq 3$, the second derivative terms become of lower order, and we can set $\e\sim 1$ to obtain \eqref{e:k3}. For $k=2$ we obtain \eqref{e:k2}.  This finishes the proof. 
\end{proof}

Finally, we have needed estimates on the full under top derivatives  $(\cL_\phi f)^{(l)}$, $k\geq 2$, with the use of above results. So, for any $k\geq 2$ and with the same convention of using $\lesssim$ up to a constant $C(\rhomin,\rhomax,|\rho'|_\infty,|f'|_\infty,\ldots , |f^{(k-1)}|_\infty, |\rho^{(k-1)}|_\infty)$,  we deduce from \lem{l:Linfty} with $\g = 0$ and $r\sim 1$,
\[
\begin{split}
|\cL_\phi (f^{(k-1)}) (x)| & \lesssim \sqrt{\dD_\a[ f^{(k)}] (x)} + |f^{(k)}(x)|+1 \\
|\cL_\phi (f^{(k-2)}) (x)| & \lesssim \sqrt{\dD_\a[ f^{(k-1)}] (x)} + 1 \\
|\cL_\phi (f^{(l)}) (x)| & \lesssim 1, \qquad 0\leq l\leq k-3 \\
\end{split}
\]
In combination with the commutator estimates established in \lem{l:kcomm} we obtain the following lemma.

\begin{lemma}\label{l:Lkbound}
	For any smooth function $f$ and $k \geq 2$, we have
	\begin{align}
	|(\cL_\phi f)^{(k-1)} (x)| & \lesssim \sqrt{\dD_\a[ f^{(k)}] (x)} + |f^{(k)}(x)| +\sqrt{\dD_\a[ f^{(k-1)}] (x)}+\sqrt{\dD_\a[ \rho^{(k-1)}] (x)}+1  \label{e:LkboundA} \\
	|(\cL_\phi f)^{(k-2)} (x)| & \lesssim \sqrt{\dD_\a[ f^{(k-1)}] (x)} + 1 \label{e:LkboundB} \\
	|(\cL_\phi f)^{(l)} (x)| & \lesssim 1, \qquad 0\leq l\leq k-3, \label{e:LkboundC}
	\end{align}
with $\lesssim$ meaning  up to a constant factor 
\[
C = C(\rhomin,\rhomax,|\rho'|_\infty,|f'|_\infty,\ldots , |f^{(k-1)}|_\infty, |\rho^{(k-1)}|_\infty).
\]
\end{lemma}

\end{document}